\documentclass[10pt,reqno]{amsart}
\usepackage{amsmath}
\usepackage{amssymb}
\usepackage{amsthm}
\usepackage[notcite,notref]{showkeys}
\usepackage[usenames]{color}
\usepackage{eepic,epic}

\newtheorem{thm}{Theorem}[section]

\newtheorem{lem}[thm]{Lemma}
\newtheorem{prop}[thm]{Proposition}

\theoremstyle{definition}

\theoremstyle{remark}
\newtheorem{rem}[thm]{Remark}
\numberwithin{equation}{section}

\newcommand{\dist}{\text{dist }}

\begin{document}

\title[]
{A priori bound for nonlinear elliptic equation and system involving a fractional Laplacian}

\author{Woocheol Choi}
\subjclass[2010]{Primary 35J60}
\address{Department of Mathematical Sciences, Seoul National University, Seoul 151-747, Korea}
\email{chwc1987@snu.ac.kr}
\thanks{The research of this paper was supported by Global Ph.D Fellowship(300-20120003)
of the government of South Korea.}
\maketitle

\begin{abstract}
In this paper we study nonlinear elliptic system involving the fractional Laplacian on bounded domains. We obtain existence and non-existence results, $a~ priori$ estimates of Gidas-Spruck type, and the symmetric property.
\end{abstract}

\section{introduction}
In this paper we consider the following nonlinear problem:
\begin{eqnarray}\label{eq-sys}
\left\{ \begin{array}{ll} \mathcal{A}_{s} u = v^p & \quad \textrm{in} ~\Omega,
\\
\mathcal{A}_{s} v = u^{q}& \quad \textrm{in} ~\Omega,
\\
u >0,~ v>0 &\quad \textrm{in}~ \Omega,
\\
u=v =0& \quad \textrm{on}~ \partial\Omega,
\end{array}\right.
\end{eqnarray}
where $0< s < 1$, $p>1$, $q >1$, $\Omega$ is a smooth bounded domain of $\mathbb{R}^n$ and $\mathcal{A}_{s}$ denotes the fractional Laplace operator $(-\Delta)^s$ in $\Omega$ with zero Dirichlet boundary values on $\partial \Omega$, defined in terms of the spectra of the Dirichlet Laplacian $-\Delta$ on $\Omega$.

The fractional Laplacian appears in diverse areas including physics, biological modeling and mathematical finances and partial differential equations involving the fractional Laplacian have attracted the attention of many researchers.
An important feature of the fractional Laplacian is its nonlocal property, which makes it difficult to handle.
Recently, Caffarelli and Silvestre \cite{CS} developed a local interpretation of the fractional Laplacian given in $\mathbb{R}^n$ by considering a Neumann type operator in the extended domain $\mathbb{R}^{n+1}_{+} := \{(x,t) \in \mathbb{R}^{n+1}: t > 0\}$.
This observation made a significant influence on the study of related nonlocal problems.
A similar extension was devised by Cabr\'e and Tan \cite{CT} and Capella, D\'{a}vila, Dupaigne, and Sire \cite{CDDS} (see Br\"andle, Colorado, de Pablo, and S\'anchez \cite{BCPS1} and Tan \cite{T1} also).

Based on these extensions, many authors studied nonlinear problems of the form $\mathcal{A}_{s} u = f(u)$, where $f:\mathbb{R}^n \rightarrow \mathbb{R}$ is a certain function. When $s= {1 \over 2}$, Cabr\'e and Tan \cite{CT} established the existence of positive solutions for equations having nonlinearities with the subcritical growth, their regularity, the symmetric property, and a priori estimates of the Gidas-Spruck type by employing a blow-up argument along with a Liouville type result for the square root of the Laplacian in the half-space. Br\"andle, Colorado, de Pablo, and S\'anchez \cite{BCPS1}  dealt with a subcritical concave-convex problem.
For $f(u) = u^q$ with the critical and supercritical exponents $q \geq \frac{n+2s}{n-2s}$, the nonexistence of solutions was proved in \cite{BCPS1, T, T1}
in which the authors devised and used the Pohozaev type identities. The Brezis-Nirenberg type problem was studied in \cite{T} for $s=1/2$. Choi, Kim, and Lee \cite{CKL} studied the asymptotic behavior of solutions to the slightly sub-critical problem and the Brezis-Nirenberg type problem for $0<s<1$.
\

When $s=1$ the nonlinear problem \eqref{eq-sys} corresponds the Lane-Emden system, which have been studied extensively by numerous authors. We refer to \cite{CFM, FF, HV, FLN} and references therein, and the book \cite{QS} for a systematic study of this subject.  

Before studying the problem \eqref{eq-sys} we provide a different proof to the $a~ priori$ estimate for solutions to the problem 
\begin{equation}\label{eq-one}
\left\{ \begin{array}{ll} 
\mathcal{A}_s u = f(u)&\quad \textrm{in}~\Omega,
\\
u>0&\quad \textrm{in}~\Omega,
\\
u=0&\quad \textrm{on}~\partial \Omega.
\end{array}
\right.
\end{equation}
\begin{thm}\label{letn2}
Let $n \geq 2$ and $0<s<1$. Assume that $\Omega \subset \mathbb{R}^n$ is a smooth  bounded domain and $f(u) = u^p, ~1 <p< \frac{n+2s}{n-2s}.$
\

Then, there exists a constant $C(p,s,\Omega)$ depending only on $p$, $s$ and $\Omega$ such that every weak solution of $(1.1)$ satisfies
\begin{eqnarray*}
\| u\|_{L^{\infty} (\Omega)} \leq C(p,s,\Omega).
\end{eqnarray*}
Moreover, the statement holds for any function $f: \mathbb{R}_+ \rightarrow \mathbb{R}$ satisfying \textbf{Condition A} (see Section 4).
\end{thm}
The result of Theorem \ref{letn2} was proved by C\'abre-Tan \cite{CT} for $s=1/2$ and Tan \cite{T1} for $1/2 < s<1$. They employed the blow-up argument combining the Liouville type results. In our proof, first we obtain uniform bounds of $L^{\infty}$ norm near the boundary and $L^p (\Omega)$ norm for solutions to \eqref{eq-one}. Next, combining this with a local Pohohzaev inequality of Proposition \ref{prop-sub-estimate}, we shall get a uniform bound for a higher norm. Then we use the Sobolev embedding interatively to obtain the $L^{\infty}$ estimate. As this approach does not require a Liouville-type result, the function $f(u)$ is not required to have a precise asymptocity as $u \rightarrow \infty$. Moreover, since this approach is more flexible to obtain $a~ priori$ estimates for the nonlinear system \eqref{eq-sys}.
%In addition, our Lemma \ref{u0in} and Lemma \ref{suppo} will be crucial to study the asymptotic behaviour of positive solutions to the problem \eqref{a12u} near the critical exponent in the forthcoming paper \cite{CL}.
In studying the nonlinear system \eqref{eq-sys} we say that a pair of exponents $(p,q)$ is sub-critical if $\frac{1}{p+1}+\frac{1}{q+1} > \frac{n-2s}{n}$, critical if  $\frac{1}{p+1}+\frac{1}{q+1} = \frac{n-2s}{n}$, and super-critical if  $\frac{1}{p+1}+\frac{1}{q+1} < \frac{n-2s}{n}$.
Then we have the following existence result. 
\begin{thm}\label{thm-ex}Suppose that $(p, q)$ is sub-critical. Let $\alpha>0$ and $\beta>0$ be such that
\[
\frac{1}{2}-\frac{1}{p+1} < \frac{ \alpha}{n},\quad \frac{1}{2} - \frac{1}{q+1} < \frac{\beta}{n}, \quad \textrm{and} \quad \alpha + \beta =2s.
\]
Then, there exists a (nontrivial) weak solution $(u,v) \in H_0^{\alpha}(\Omega)\times H_0^{\beta}(\Omega)$ to the problem \eqref{12uv}.
\end{thm}
%The analogue problem to \eqref{12uvv} for the Laplacian is the well-known Lane-Emden system, which has been investigated widely in the last decades (see \cite{QS} and references therein). The main result of our study is the following theorem on a priori estimates of this system.
See Section 2 for the definition of weak solution and the space $H^{\alpha}_0 (\Omega)$. We shall obtain a Pohozaev type identity to obtain the following non-existence result.
\begin{thm}\label{thm-noex} Assume that the domain $\Omega$ is bounded and starshaped. Take $p>1$ and $q >1$ such that $(p,q)$ is critical or sub-critical, i.e.,  
\begin{eqnarray*}
\frac{1}{p+1} + \frac{1}{q+1} \leq \frac{n-2s}{n}.
\end{eqnarray*}
Then \eqref{eq-sys} does not has a bounded weak solution.
\end{thm}
Next we state the symmetric property of solutions.
\begin{thm}\label{thm-syme}
Suppose that a bounded smooth domain $\Omega \subset \mathbb{R}^n$ is convex in the $x_1$-direction and symmetric with respect to the hyperplane $\{x_1 =0\}$. Let $(u,v)$ be a $C^2 (\bar{\Omega})$ solution of \eqref{12uv}.
\

Then, the functions $u$ and $v$ are symmetric in $x_1$-direction, that is, $u(-x_1,x') = u(x_1,x')$, $v(x_1, x') = v(-x_1, x')$ for all $(x_1,x' ) \in \Omega$. Moreover we have $\frac{\partial u}{\partial x_1} < 0$ and $\frac{\partial v}{\partial x_1} < 0$ for $x_1 >0$.
\end{thm}
Finally we shall establish $a~priori$ estimates of Gidas-Spruck type.
\begin{thm}\label{letn22}
Assume that $\Omega \subset \mathbb{R}^n$ is a smooth convex bounded domain and $p>1$ and $q >1$ are such that $(p,q)$ is sub-critical.
%\begin{eqnarray*}
%\frac{1}{p+1} + \frac{1}{q+1} > \frac{n-2s}{n}.
%\end{eqnarray*}
Then, there exists a constant $C(p,q,\Omega)$, which depends only on $p$ and $\Omega$, such that every weak solution of \eqref{eq-sys} satisfies
\begin{eqnarray*}
\| u\|_{L^{\infty} (\Omega)} + \| v\|_{L^{\infty}(\Omega)} \leq C(p,q, \Omega).
\end{eqnarray*}
\end{thm}
The rest of this paper is organized as follows. In Section 2 we briefly review the basic results concerning the fractional Laplacian. In section 3, we shall prove two important estimates for solutions to \eqref{eq-sys} and \eqref{eq-one} from the Pohozaev identities. These estimates will be used importantly in the proofs of $a~priori$ estimates. In Section 4 we prove Theorem \ref{letn2}. The nonlinear system \eqref{eq-sys} will be studied throughout Section 5. First we establish the existence and the non-existence results of Theorem \ref{thm-ex} and Theorem \ref{thm-noex}. Then we obtain a Brezis-Kato type result and study the regularity of solutions to \eqref{eq-sys}. Next, we establish a moving plane argument to prove Theorem \ref{thm-syme}. Finally, we shall prove Theorem \ref{letn22}.

%We shall use the notation $\lesssim $ when the constant of an estimate $A \leq CB$ depends only on the fucntion $f$, the domain $\Omega$ and some parameters fixed in the proofs.
\section{Preliminaries}

In this section we first recall the backgrounds of the fractional Laplacian. We refer to \cite{BCPS1, CT, CS, CDDS, T1} for the details.
\subsection{Fractional Sobolev spaces, fractional Laplacians and $s$-harmonic extensions}\label{subsec_frac_Sob}
Let $\Omega$ be a smooth bounded domain of $\mathbb{R}^n$.
Let also $\{ \lambda_k, \phi_k\}_{k=1}^{\infty}$ be a sequence of the eigenvalues and corresponding eigenvectors of the Laplacian operator $-\Delta$ in $\Omega$ with the zero Dirichlet boundary condition on $\partial \Omega$,
\[\left\{ \begin{array}{ll}
- \Delta \phi_k = \lambda_k \phi_k &\text{in}~ \Omega,\\
\phi_k = 0 &\text{on}~ \partial \Omega,\\
\end{array}\right.\]
such that $\|\phi_k\|_{L^2(\Omega)} = 1$ and $\lambda_1 < \lambda_2 \le \lambda_3 \le \cdots$.
Then we set the fractional Sobolev space $H_0^s (\Omega)$ $(0 < s < 1)$ by
\begin{equation}\label{H_0^s}
H_0^s (\Omega) = \left\{ u = \sum_{k=1}^{\infty} a_k \phi_k \in L^2 (\Omega) : \sum_{k=1}^{\infty} a_k^2 \lambda_k^{s} < \infty \right\},
\end{equation}
which is a Hilbert space whose inner product is given by
\[\left\langle \sum_{k=1}^{\infty} a_k \phi_k, \sum_{k=1}^{\infty} b_k \phi_k \right\rangle_{H_0^s(\Omega)} = \sum_{k=1}^{\infty} a_k b_k\lambda_k^s \qquad \text{if } \sum_{k=1}^{\infty} a_k \phi_k,\ \sum_{k=1}^{\infty} b_k \phi_k \in H_0^s(\Omega). \]
Moreover, for a function in $H_0^{s}(\Omega)$, we define the fractional Laplacian $\mathcal{A}_{s} : H_0^s (\Omega) \rightarrow H_0^s (\Omega) \simeq H_0^{-s} (\Omega)$ as
\[\mathcal{A}_{s} \left ( \sum_{k=1}^{\infty} a_k \phi_k \right) = \sum_{k=1}^{\infty} a_k \lambda_k^s \phi_k.\]
We also consider the square root $\mathcal{A}_{s}^{1/2}: H_0^s(\Omega) \to L^2(\Omega)$ of the positive operator $\mathcal{A}_{s}$ which is in fact equal to $\mathcal{A}_{s/2}$. Note that by the above definitions, we have
\[\left\langle u, v \right\rangle_{H_0^s(\Omega)} = \int_{\Omega} \mathcal{A}_{s}^{1/2}u \cdot \mathcal{A}_{s}^{1/2}v = \int_{\Omega} \mathcal{A}_{s} u \cdot v \quad \text{for } u, v \in H_0^s(\Omega). \]

\medskip
Regarding \eqref{u0inc} (see also \eqref{s-extension} below), we need to introduce some more function spaces on $\mathcal{C} = \Omega \times (0, \infty)$ where $\Omega$ is either a smooth bounded domain.
If $\Omega$ is bounded, the function space $H_{0,L}^s(\mathcal{C})$ is defined as the completion of
\[ C_{c,L}^{\infty}(\mathcal{C})
:= \left\{ U \in C^{\infty}\left(\overline{\mathcal{C}}\right) : U = 0 \text{ on } \partial_L\mathcal{C} = \partial \Omega \times (0,\infty) \right\}\]
with respect to the norm
\begin{equation}\label{weighted_norm}
\|U\|_{\mathcal{C}} = \left( \int_{\mathcal{C}} t^{1-2s} |\nabla U|^2 \right)^{1 \over 2}.
\end{equation}
Then it is a Hilbert space endowed with the inner product
\[(U,V)_{\mathcal{C}} = \int_{\mathcal{C}} t^{1-2s} \nabla U \cdot \nabla V \quad \text{for} \quad U, \ V \in H_{0,L}^{s}(\mathcal{C}).\]
Recall that if $\Omega$ is a smooth bounded domain, it is verified that
\begin{equation}\label{eq_Sobo_trace}
H_0^s(\Omega) = \{u = \text{tr}|_{\Omega \times \{0\}}U: U \in H^s_{0,L}(\mathcal{C})  \}
\end{equation}
in \cite[Proposition 2.1]{CS} and \cite[Proposition 2.1]{CDDS} and \cite[Section 2]{T1}.
Furthermore, it holds that
\[\|U(\cdot, 0)\|_{H^s(\mathbb{R}^n)} \le C \|U\|_{\mathbb{R}^{n+1}_+}\]
for some $C > 0$ independent of $U \in \mathcal{D}^s(\mathbb{R}^{n+1}_+)$.

\medskip
Now we may consider the fractional harmonic extension of a function $u$ defined in $\Omega$, where $\Omega$ is a smooth bounded domain.
By the celebrated results of Caffarelli-Silvestre \cite{CS} (for $\mathbb{R}^n$) and Cabr\'e-Tan \cite{CT} (for bounded domains, see also \cite{CDDS, BCPS1, T1}),
if we set $U \in H^s_{0,L}(\mathcal{C})$ (or $\mathcal{D}^s(\mathbb{R}^{n+1}_+)$) as a unique solution of the equation
\begin{equation}\label{s-extension}
\left\{ \begin{array}{ll} \text{div}(t^{1-2s} \nabla U) = 0 &~\text{in}~ \mathcal{C},\\
U = 0 & ~\text{on}~ \partial_L \mathcal{C},\\
U(x,0)= u(x) &~ \text{for}~ x \in \Omega,
\end{array}\right.
\end{equation}
for some fixed function $u \in H^s_0(\Omega)$ (or $H^s(\mathbb{R}^n)$),
then $\mathcal{A}_{s}u = \partial_{\nu}^s U|_{\Omega \times \{0\}}$ where the operator $u \mapsto \partial_{\nu}^s U|_{\Omega \times \{0\}}$ is defined in \eqref{pns}. 
Consequently, we study the following type local problem on a half-cylinder $\mathcal{C}:= \Omega \times [0, \infty)$,
\begin{equation}\label{u0inc}
\left\{ \begin{array}{ll} \text{div}(t^{1-2s} \nabla U )= 0&\quad \text{in}~ \mathcal{C} = \Omega \times (0,\infty),\\
U = 0 &\quad \text{on} ~\partial_L \mathcal{C} : = \partial \Omega \times (0, \infty),\\
\partial_{\nu}^{s} { U} =g(x) & \quad \text{on} ~\Omega \times \{0\},
\end{array}
\right.
\end{equation}
where $\nu$ is the outward unit normal vector to $\mathcal{C}$ on $\Omega \times \{0\}$ and
\begin{equation}\label{pns}
\partial_{\nu}^{s}U(x,0):= -C_s^{-1} \left(\lim_{t \rightarrow 0+} t^{1-2s} \frac{\partial U}{\partial t}(x,t)\right) \quad \text{for } x \in \Omega
\end{equation}
where $C_s:={2^{1-2s}\Gamma(1-s)}/{\Gamma(s)}$. Under appropriate regularity assumptions, the trace of a solution $U$ of \eqref{u0inc} on $\Omega \times \{0\}$ solves the nonlinear problem 
\begin{equation}
\left\{\begin{array}{ll} \mathcal{A}_s u (x) = g(x)&\quad \textrm{in}~\Omega,
\\
u=0&\quad \textrm{on}~\partial \Omega.
\end{array}
\right.
\end{equation} 
For $n \geq 2$, we have the following Sobolev trace inequality 
\begin{eqnarray*}
\left( \int_{\Omega} |w(x,0)|^{2n/(n-1)} dx \right)^{(n-1)/2n} \leq C \left(\int_{\mathcal{C}} |\nabla w (x,y)|^2 dx dy\right)^{1/2} \quad \forall w \in H_{0,L}^{1}(\mathcal{C}),
\end{eqnarray*}
where the constant $C>0$ depends on the dimension $n$.
We set $\textrm{tr}_{\Omega}$ be the trace operator on $\Omega \times \{0\}$ for functions in $H_{0,L}^{1}(\mathcal{C})$:
\begin{eqnarray*}
\textrm{tr}_{\Omega} v : = v(x,0)~\textrm{for}~v \in H_{0,L}^{1}(\mathcal{C}).
\end{eqnarray*}
By weak solutions, we mean the following:
Let $g \in L^{\frac{2N}{N+2s}}(\Omega)$. Given the problem
\begin{equation}\label{eq-weak1}
\left\{\begin{array}{ll}
\mathcal{A}_s u = g(x) &\quad \textrm{in}~\Omega,
\\
u = 0 &\quad \textrm{on}~\partial\Omega,
\end{array}\right.
\end{equation}
we say that a function $u \in H_0^{s}(\Omega)$ is a weak solution of \eqref{eq-weak1} provided
\begin{equation}
\int_{\Omega} \mathcal{A}_s^{1/2} u \cdot \mathcal{A}_s^{1/2} \phi \,dx = \int_{\Omega} g(x)\phi(x)\,dx
\end{equation}
for all $\phi \in H^s_0(\Omega)$.
Also, given the problem
\begin{equation}\label{eq-weak2}
\left\{
\begin{array}{ll}
\textrm{div}(t^{1-2s} \nabla U) = 0 &\quad \textrm{in}~\mathcal{C},
\\
U = 0 &\quad \textrm{on}~\partial_{L}\mathcal{C},
\\
\partial_{\nu}^{s} U = g(x)&\quad \textrm{on}~\Omega \times\{0\},
\end{array}
\right.
\end{equation}
we say that a function $U \in H_0^1(t^{1-2s}, \mathcal{C})$ is a weak solution of \eqref{eq-weak2} provided
\begin{equation}
\int_{\mathcal{C}} t^{1-2s}  \nabla U(x,t) \cdot \nabla \Phi(x,t)\, dx dt = C_s\int_{\Omega} g(x)\Phi(x,0)\,dx
\end{equation}
for all $\Phi \in H_0^1(t^{1-2s}, \mathcal{C})$.
We have the following trace inequality.
\begin{equation}\label{eq-sharp-trace}
\left( \int_{\Omega} |U(x,0)|^{2^*(s)} dx \right)^{\frac{1}{2^*(s)}} \leq \frac{\mathcal{S}_{N,s}}{\sqrt{C_s}} \left( \int_{\mathcal{C}} t^{1-2s} |\nabla U(x,t)|^2 dx dt \right)^{1 \over 2},
\quad U \in H^1_0(t^{1-2s}, \mathcal{C}).
\end{equation}
Next we state the embedding result.\begin{lem}[see \cite{CS}]\label{lem-trace}
Let $w \in L^{p}(\Omega)$  for some $p < \frac{N}{2s}$. Assume that $U$ is a weak solution of the problem
\begin{equation}\label{eq-lem-basic}
\left\{\begin{array}{ll}
\textrm{\em div}(t^{1-2s} \nabla U) = 0&\quad \textrm{in}~\mathcal{C},
\\
U=0&\quad \textrm{on}~\partial_{L}\mathcal{C},
\\
\partial_{\nu}^s U = w &\quad \textrm{on}~\Omega \times \{0\}.
\end{array}
\right.
\end{equation}
Then we have
\begin{equation}
\left\| U(\cdot, 0)\right\|_{L^q (\Omega)} \leq C_{p,q} \left\| w \right\|_{L^p (\Omega)},
\end{equation}
for any $q$ such that $\frac{N}{q} \leq \frac{N}{p} -2s$.
\end{lem}
\begin{proof}
We multiply \eqref{eq-lem-basic} by $|U|^{\beta-1}U$ for some  $\beta >1$ to get
\begin{equation}
 \int_{\Omega} w(x) |U|^{\beta-1} U (x,0)\, dx= \beta \int_{\mathcal{C}} t^{1-2s} |U|^{\beta-1}|\nabla U|^2\, dxdt.
\end{equation}
Then, applying the trace embedding \eqref{eq-sharp-trace} and H\"older's inequality we can observe
\begin{equation}\label{eq-lem-beta}
\left\| |U|^{\frac{\beta+1}{2}}(\cdot,0) \right\|^2_{L^{\frac{2N}{N-2s}}(\Omega)} \leq C_\beta \left\| |U|^{\beta}(\cdot,0)\right\|_{L^{\frac{\beta+1}{2\beta}\cdot \frac{2N}{N-2s}}} \left\| w\right\|_{p},
\end{equation}
where $p$ satisfies $\frac{1}{p} + \frac{(N-2s)\beta}{N(\beta+1)} =1$. Let $q= \frac{N(\beta+1)}{N-2s}$, %It holds that $n/q = n/p -2s$, and
then \eqref{eq-lem-beta} gives the desired inequality.
\end{proof}
%\begin{proof}
%We multiply \eqref{eq-lem-basic} by $|U|^{\beta-1}U$ for some  $\beta >1$ to get
%\begin{equation}
% \int_{\Omega} w(x) |U|^{\beta-1} U (x,0)\, dx= \beta \int_{\mathcal{C}} t^{1-2s} |U|^{\beta-1}|\nabla U|^2\, dxdt.
%\end{equation}
%Then, applying the trace embedding \eqref{eq-sharp-trace} and H\"older's inequality we can observe
%\begin{equation}\label{eq-lem-beta}
%\left\| |U|^{\frac{\beta+1}{2}}(\cdot,0) \right\|^2_{L^{\frac{2N}{N-2s}}(\Omega)} \leq C_\beta \left\| |U|^{\beta}(\cdot,0)\right\|_{L^{\frac{\beta+1}{2\beta}\cdot \frac{2N}{N-2s}}} \left\| w\right\|_{p},
%\end{equation}
%where $p$ satisfies $\frac{1}{p} + \frac{(N-2s)\beta}{N(\beta+1)} =1$. Let $q= \frac{N(\beta+1)}{N-2s}$, %It holds that $n/q = n/p -2s$, and
%then \eqref{eq-lem-beta} gives the desired inequality.
%\end{proof}

\section{A local Pohozaev inequality}
In this section, we prove a useful inequality satisfied by solutions to \eqref{eq-sys}, which will be crucially used in the proof of Theorem \ref{letn2} and Theorem \ref{letn22}. For each $r>0$ we set $\mathcal{I}(\Omega,{r}) = \{ x  \in \Omega : \textrm{dist}(x,\partial \Omega) \geq r \}$ and $\mathcal{O}(\Omega,r) = \{x \in \Omega : \textrm{dist}(x, \partial \Omega ) < r\}$. Then we have the following results.
\begin{prop}\label{prop-sub-estimate}\mbox{~}
\begin{enumerate}
\item Suppose that $U \in {H^s_{0,L}(\mathcal{C})}$ is a solution of the problem \eqref{u0inc} with $f$ such that $f = F'$ for a function $F \in C^1 (\mathbb{R})$.
Then, for each $ \delta >0$ and $q >\frac{n}{s}$ there is a constant $C= C(\delta, q) >0$ such that
\begin{equation}\label{eq_local_poho}
\begin{aligned}
&\min_{r \in [\delta, 2\delta]}\left|n \int_{\mathcal{I}(\Omega,{r/2})\times\{0\}} F(U) dx - \left(\frac{n-2s}{2}\right) \int_{ \mathcal{I}(\Omega,{r/2}) \times \{0\}} U f(U) dx\right|
\\
&\leq C \left[ \left(\int_{\mathcal{O}(\Omega,{2\delta})\times \{0\}} |f(U)|^{q} dx\right)^{2 \over q} + \int_{\mathcal{O}(\Omega,2\delta)\times \{0\}} |F(U)|dx + \left( \int_{\mathcal{I}(\Omega, \delta/2) \times \{0\}} |f(U)| dx \right)^2 \right].
\end{aligned}
\end{equation}
\item
Suppose that $U \in {H^s_{0,L}(\mathcal{C})}$ is a solution of the problem \eqref{u0inc} with $f$ such that $f = F'$ for a function $F \in C^1 (\mathbb{R})$.
Then, for each $ \delta >0$ and $q >\frac{n}{s}$ there is a constant $C= C(\delta, q) >0$ such that
\begin{equation}\label{eq_local_poho_2}
\begin{aligned}
&\min_{r \in [\delta, 2\delta]}\left|n \int_{\mathcal{I}(\Omega,{r/2})\times\{0\}} \left[ F(U)+G(V)\right]~ dx - \int_{ \mathcal{I}(\Omega,{r/2}) \times \{0\}}\left[\left(\frac{n-2s}{2}- \theta \right) U f(U) + \theta V g(V) \right]~dx \right|
\\
&\leq C \left(\int_{\mathcal{O}(\Omega,{2\delta})\times \{0\}} (|f(U)|+|g(V)|)^{q} dx\right)^{2 \over q} + \int_{\mathcal{O}(\Omega,2\delta)\times \{0\}} |F(U)|+|G(V)| dx 
\\
&\quad\quad\quad\quad\quad\quad\quad\quad\quad\quad\quad\quad\quad\quad\quad\quad\quad\quad\quad\quad\quad\quad\quad\quad + \left( \int_{\mathcal{I}(\Omega, \delta/2) \times \{0\}} |f(U)|+|g(V)| dx \right)^2.
\end{aligned}
\end{equation}
\end{enumerate}
\end{prop}
\begin{rem}\label{rem-onetwo} The statement (1) of Proposition \ref{prop-sub-estimate} was proved in \cite{CKL}. We note that a solution $u$ to \eqref{eq-one} with $f(x) =x^p$ satisfies $(u,u)$ satisfies \eqref{eq-sys} with $q=p$. Thus the statement (1) follows directly from the statement (2) in Proposition \ref{prop-sub-estimate}.
\end{rem}

\begin{proof}
By the above remark, it sufices to prove \eqref{eq_local_poho_2}. By a direct computation, we have the following identity
\begin{equation}\label{eq-prop-id}
\textrm{div}[t^{1-2s}(z,\nabla v) \nabla u +t^{1-2s}(z,\nabla u) \nabla v] - \textrm{div}(t^{1-2s}z(\nabla u \cdot \nabla v)) + (n-2s)t^{1-2s} \nabla u \cdot \nabla v =0.
\end{equation}
For a given set $A \in \mathcal{C}$, using integration by parts we have
\begin{equation}\label{eq-3-1}
\begin{split}
\int_{A} t^{1-2s} \nabla u \cdot \nabla v~ dxdt&=  \int_{\partial^{+} A} t^{1-2s} (\nabla u, \nu) v~ dS + \int_{\partial_b A} \partial_{\nu}^s u~ v(x) dx
\\
&=\int_{\partial^{+}A} t^{1-2s}(\nabla v,\nu) u ~dS + \int_{\partial_b A} \partial_{\nu}^s u~ v (x)dx.
\end{split}
\end{equation}
Also we have
\begin{equation}\label{eq-3-2}
\begin{split}
&\int_A \textrm{div} \left[ t^{1-2s}(z,\nabla v) \nabla u + t^{1-2s} (z,\nabla u) \nabla v\right] dx dt
\\
&\quad\quad = \int_{\partial^{+}A} \left[ t^{1-2s} (z,\nabla v) (\nabla u, \nu) + t^{1-2s}(z,\nabla u) (\nabla v, \nu) \right] dS + \int_{\partial_b A}(x, \nabla_x v) \partial_{\nu}^s u + (x,\nabla_x u) \partial_{\nu}^s v dx,
\end{split}
\end{equation}
and
\begin{equation}\label{eq-3-3}
\int_{A} \textrm{div}(t^{1-2s} z (\nabla u \cdot \nabla v)) = \int_{\partial^{+}A} t^{1-2s} (z,\nu) (\nabla u \cdot \nabla v) dS.
\end{equation}
We define the following sets:
\begin{align*}
D_r &= \left\{ z \in \mathbb{R}^{n+1}_{+} : \textrm{dist}(z,\mathcal{I}(\Omega,r) \times \{0\}) \leq r/2\right\},\\
\partial D_r^{+} &= \partial D_r \cap \left\{(x,t) \in \mathbb{R}^{n+1}: t>0\right\} \quad \text{and} \quad E_{\delta} = \bigcup_{r = \delta}^{2\delta} \partial D_{r}^{+}.
\end{align*}
Note that $\partial D_r = \partial D_r^{+} \cup (\mathcal{I}(\Omega,r/2)\times \{0\})$.
Fix a small number $\delta >0$ and a value $\theta >0$.
We integrate the identity \eqref{eq-prop-id} over $D_r$ for each $r \in (0,2\delta]$ to derive 
\begin{equation}
\begin{split}
&\theta \int_{\mathcal{I}(\Omega, r/2)\times\{0\}} \partial_{\nu}^s u \cdot v dx + (n-2s-\theta) \int_{\mathcal{I}(\Omega, r/2)\times\{0\}} \partial_{\nu}^s v \cdot u dx
\\
&\qquad\qquad\qquad\qquad\qquad\qquad\qquad\qquad\qquad + \int_{\mathcal{I}(\Omega, r/2)\times\{0\}} \left[(x,\nabla_x v) \partial_{\nu}^s u + (x,\nabla_x u) \partial_{\nu}^s v \right] dx
\\
&= -\theta \int_{\partial D_r^{+}} t^{1-2s} (\nabla u, \nu) v dS  - (n-2s-\theta) \int_{\partial D_r^{+}} t^{1-2s} (\nabla v, \nu) u dS 
\\
&\quad  + \int_{\partial D_r^{+}} t^{1-2s}(z,\nu) (\nabla u \cdot \nabla v) dS - \int_{\partial D_r^{+}} \left[ t^{1-2s} (z,\nabla v)(\nabla u, \nu) + t^{1-2s}(z,\nabla u)(\nabla v, \nu) \right] dS,
\end{split}
\end{equation}
where \eqref{eq-3-1}, \eqref{eq-3-2}, and \eqref{eq-3-3} are used.
%In view of Lemmas 4.4 and 4.5 of \cite{CS2}, one can deduce that the $i$-th component $\partial_{x_i} U$ of $\nabla_x U$ is H\"older continuous in $\overline{D_r}$ for each $i = 1, \cdots, n$, which justifies the above formula.
By using $\partial_{\nu}^s U = f(V)$, $\partial_{\nu}^s V= g(V)$ and performing integration by parts, we derive
\begin{equation*}
\begin{split}
&\theta \int_{\mathcal{I}(\Omega, r/2)\times\{0\}} g(v) \cdot v dx + (n-2s-\theta) \int_{\mathcal{I}(\Omega, r/2)\times\{0\}} f(u) \cdot u dx - \int_{\mathcal{I}(\Omega, r/2)\times\{0\}} \left[n F(u) + n G(v) \right] dx
\\
&= -\theta \int_{\partial D_r^{+}} t^{1-2s} (\nabla u, \nu) v dS  - (n-2s-\theta) \int_{\partial D_r^{+}} t^{1-2s} (\nabla v, \nu) u dS 
\\
&\quad  + \int_{\partial D_r^{+}} t^{1-2s}(z,\nu) (\nabla u \cdot \nabla v) dS - \int_{\partial D_r^{+}} \left[ t^{1-2s} (z,\nabla v)(\nabla u, \nu) + t^{1-2s}(z,\nabla u)(\nabla v, \nu) \right] dS
\\
&\quad + \int_{\partial\mathcal{I}(\Omega, r/2)\times\{0\}} (x,\nu)( F(u) + G(v)) dS.
\end{split}
\end{equation*}
From this identity we get
\begin{multline*}
\left| \int_{\mathcal{I}(\Omega, r/2)\times\{0\}} \left[\theta g(v) \cdot v - n G(v)\right] dx +  \int_{\mathcal{I}(\Omega, r/2)\times\{0\}} \left[(n-2s-\theta) f(u) \cdot u - n F(u)\right] dx\right|
\\
\leq C \int_{\partial D_r^{+}} t^{1-2s} (|\nabla U|^2 + U^2 + |\nabla V|^2 + V^2) dS + \int_{\partial \mathcal{I}(\Omega,{r/2}) \times\{0\}} \langle x,\nu \rangle  (F(U)+G(V)) dS_x.
\end{multline*}
We integrate this identity with respect to $r$ over an interval $[\delta, 2\delta]$ and then use the Poincar\'e inequality. Then we observe
\begin{multline*}
\min_{r \in [\delta ,2\delta]} \left| \int_{\mathcal{I}(\Omega, r/2)\times\{0\}} \left[\theta g(v) \cdot v - n G(v)\right] dx +  \int_{\mathcal{I}(\Omega, r/2)\times\{0\}} \left[(n-2s-\theta) f(u) \cdot u - n F(u)\right] dx\right|
\\
\leq C \int_{E_{\delta}}t^{1-2s}  (|\nabla U|^2 + U^2 + |\nabla V|^2 + V^2)  dz  + C \int_{\mathcal{O}(\Omega,{\delta})}|F(U)(x,0)|+|G(V)(x,0)| dx.
\end{multline*}
We only need to estimate the first term of the right-hand side of the previous inequality since the second term is already one of the terms which constitute the right-hand side of \eqref{eq_local_poho}. Note that
\begin{equation}\label{eq-U-decom}
\nabla_z U(z) = \int_{\Omega} \nabla_z  G_{\mathbb{R}^{n+1}_{+}} (z,y) f(U)(y,0) dy - \int_{\Omega} \nabla_z H_{\mathcal{C}}(z,y) f(U)(y,0) dy
\end{equation}
for $z \in E_{\delta}$.

Let us deal with the last term of \eqref{eq-U-decom} first.
Admitting the estimation
\begin{equation}\label{eq-appendix-H-pre}
\sup_{y \in \Omega} \int_{E_{\delta}} t^{1-2s} |\nabla_z H_{\mathcal{C}}(z,y)|^2 dz \le C
\end{equation}
for a while and using H\"older's inequality, we get
\begin{equation}\label{eq-appendix-H}
\begin{aligned}
&\ \int_{E_{\delta}} t^{1-2s} \left( \int_{\Omega} |\nabla_z H_{\mathcal{C}}(z,y) f(U)(y,0)| dy \right)^2 dz
\\
&\le \left(\sup_{y \in \Omega} \int_{E_{\delta}} t^{1-2s} |\nabla_z H_{\mathcal{C}}(z,y)|^2 dz\right) \left( \int_{\Omega} |f(U)(y,0)| dy\right)^2 \le C \left( \int_{\mathcal{I}(\Omega,{\delta}) \cup \mathcal{O}(\Omega,{\delta})} |f(U)(y,0)| dy\right)^2
\\
&\le C \left[\left(\int_{\mathcal{O}(\Omega,{2\delta})} |f(U)(y,0)|^q dy\right)^{2 \over q} + \left( \int_{\mathcal{I}(\Omega, \delta/2) } |f(U)(y,0)| dy \right)^2\right],
\end{aligned}
\end{equation}
which is a part of the right-hand side of \eqref{eq_local_poho}.

The validity of \eqref{eq-appendix-H-pre} can be reasoned as follows.
First of all, if $y$ is a point in $\Omega$ such that $\dist(y, E_{\delta}) \leq \delta/ 2$,
then it automatically satisfies that $\dist (y, \partial \Omega) \geq \delta/2$ from which we know
\[\sup_{\dist(y,\partial \Omega) \ge \delta/2} \left(\int_{E_{\delta}} t^{1-2s} |\nabla_z H_{\mathcal{C}}(z,y)|^2 dz\right) \le \sup_{\dist(y, \partial \Omega) \geq \delta/2} \left(\int_{\mathcal{C}} t^{1-2s} |\nabla_z H_{\mathcal{C}} (z,y) |^2 dz\right) \le C.\]
See the proof of Lemma 2.2 in \cite{CKL} for the second inequality.
Meanwhile, in the complementary case $\dist (y, E_{\delta}) > \delta /2$, we can assert that
\begin{equation}\label{eq-h-nabla}
\int_{E_{\delta}} t^{1-2s} |\nabla_z H_{\mathcal{C}}(z,y)|^2 dz
\leq C \left(\int_{N(E_{\delta}, \delta/4)} t^{1-2s} |H_{\mathcal{C}}(z,y)|^2 dz\right)
\end{equation}
where $N(E_{\delta}, \delta/4) := \{ z \in \mathcal{C}: \dist (z, E_{\delta}) \leq \delta /4\}$.
To show this, we recall that $H_{\mathcal{C}}$ satisfies
\begin{equation}\label{eq-appendix-h}
\left\{\begin{array}{ll}
\textrm{div}(t^{1-2s} \nabla H_{\mathcal{C}}(\cdot, y)) = 0 &\quad \textrm{in}~\mathcal{C},
\\
\partial_{\nu}^{s} H_{\mathcal{C}}(\cdot, y) = 0 &\quad \textrm{on}~\Omega \times \{0\}.
\end{array}\right.
\end{equation}
Fix a smooth function $\phi \in C^{\infty}_{0} (N(E_{\delta}, \delta/4))$ such that $\phi =1$ on $E_{\delta}$ and $|\nabla \phi|^2 \leq C_0 \phi$ holds for some $C_0 > 0$,
and multiply $H_{\mathcal{C}}(\cdot,y) \phi (\cdot)$ to \eqref{eq-appendix-h}.
Then we have
\[\int_{\mathcal{C}} t^{1-2s}  | \nabla H_{\mathcal{C}} (z,y)|^2 \phi (z) + \int_{\mathcal{C}} t^{1-2s}[\nabla H_{\mathcal{C}}(z,y) \cdot \nabla \phi (z)] H_{\mathcal{C}}(z,y) dz = 0.\]
From this we deduce that
\begin{align*}
&\ \int_{\mathcal{C}} t^{1-2s}  | \nabla H_{\mathcal{C}} (z,y)|^2 \phi (z) dz
\\
&= - \int_{\mathcal{C}} t^{1-2s}[\nabla H_{\mathcal{C}}(z,y) \cdot \nabla \phi (z)] H_{\mathcal{C}}(z,y) dz
\\
& \leq \frac{1}{2 C_0} \int_{\mathcal{C}}t^{1-2s} |\nabla H_{\mathcal{C}} (z,y)|^2 |\nabla \phi (z)|^2 dz + 2C_0 \int_{N(E_{\delta}, \delta/4)} t^{1-2s} |H_{\mathcal{C}}(z,y)|^2 dz.
\end{align*}
Using the property $|\nabla \phi|^2 \leq C_0 \phi$ we derive that
\[\int_{\mathcal{C}} t^{1-2s}  | \nabla H_{\mathcal{C}} (z,y)|^2 \phi (z) dz \leq 4C_0 \int_{N(E_{\delta}, \delta/4)} t^{1-2s} |H_{\mathcal{C}}(z,y)|^2 dz.\]
It verifies inequality \eqref{eq-h-nabla}.
Since the assumption $\dist (y, E_{\delta}) > \delta /2$ implies $\dist(y, N(E_{\delta}, \delta/4)) > \delta /4$, it holds
\[\sup_{\dist (y, E_{\delta}) > \delta/2} \sup_{z \in N (E_{\delta},\delta/4)} |H_{\mathcal{C}}(z,y)|
\le \sup_{\dist (y, E_{\delta}) > \delta/2} \sup_{z \in N (E_{\delta},\delta/4)} |G_{\mathbb{R}^{n+1}_{+}} (z,y)| \leq C.\]
Combination of this and \eqref{eq-h-nabla} gives
\[\sup_{\dist(y,E_{\delta}) > \delta/2} \left(\int_{E_{\delta}} t^{1-2s} |\nabla_z H_{\mathcal{C}}(z,y)|^2 dz\right) \le
C \left(\int_{N(E_{\delta}, \delta/4)} t^{1-2s} dz\right) \le C.\]
This concludes the derivation of the desired uniform bound \eqref{eq-appendix-H-pre}.

It remains to take into consideration of the first term of \eqref{eq-U-decom}. We split the term as
\begin{align*}
&\ \int_{\Omega} \nabla_z G_{\mathbb{R}^{n+1}_{+}} (z,y) f(U) (y,0) dy \\
&= \int_{\mathcal{O}(\Omega,2\delta)} \nabla_z G_{\mathbb{R}^{n+1}_{+}} (z,y) f(U)(y,0) dy + \int_{\mathcal{I}(\Omega,{2\delta})} \nabla_z G_{\mathbb{R}^{n+1}_{+}} (z,y) f(U)(y,0) dy\\
&:= A_1 (z) + A_2 (z).
\end{align*}
Take $q > \frac{n}{s}$ and  $r>1$ satisfying $\frac{1}{q}+ \frac{1}{r} =1$. Then
\[|A_{1}(z)| \leq \left(\int_{\mathcal{O}(\Omega,2\delta)} |\nabla_z G_{\mathbb{R}^{n+1}_{+}}(z,y)|^{r} dy \right)^{1 \over r} \| f(U)(\cdot,0)\|_{L^{q}(\mathcal{O}(\Omega, 2\delta))}.\]
In light of the definition of $G_{\mathbb{R}^{n+1}_{+}}$, it holds that
\begin{align*}
\left(\int_{\mathcal{O}(\Omega,2\delta)} |\nabla_z G_{\mathbb{R}^{n+1}_{+}}(z,y)|^{r} dy\right)^{1 \over r} &\leq C \left(\int_{\mathcal{O}(\Omega,2\delta)}  \frac{1}{ |(x-y,t)|^{(n-2s+1)r}} dy\right)^{1 \over r}
\\
& \leq C \max\left\{ t^{{n \over r}-(n-2s+1)}, 1\right\}
= C\max \left\{ t^{- \frac{n}{q} +2s-1}, 1 \right\}.
\end{align*}
Thus we have
\[|A_1 (z)| \leq C \max\left\{ t^{-\frac{n}{q}+2s-1} ,1\right\}\| f(U)(\cdot,0)\|_{L^{q}(\mathcal{O}(\Omega,{2\delta}))}.\]
Using this we see
\begin{equation}\label{eq-appendix-G-1}
\begin{aligned}
\int_{E_{\delta}} t^{1-2s}|A_1(z)|^2 dz &\leq C \int_{0}^{1} \max\left\{ t^{1-2s} t^{-\frac{2n}{q}+4s-2}, t^{1-2s}\right\} \| f(U)(\cdot,0)\|_{L^{q}(\mathcal{O}(\Omega,{2\delta}))}^2 dt
\\
&  = \int_{0}^1 \max\left\{ t^{2s- \frac{2n}{q} -1}, t^{1-2s}\right\} \| f(U)(\cdot,0)\|_{L^{q}(\mathcal{O}(\Omega,{2\delta}))}^2 dt.
\\
& \leq C \|f(U)(\cdot,0)\|_{L^{q}(\mathcal{O}(\Omega,{2\delta}))}^2.
\end{aligned}
\end{equation}
Concerning the term $A_2$, we note that $E_{\delta}$ is away from $\mathcal{I}(\Omega,2\delta) \times \{0\}$. Thus we have
\[\sup_{z \in E_{\delta}, y \in  \mathcal{I}(\Omega,{2\delta})} |\nabla_z G_{\mathbb{R}^{n+1}_{+}} (z,y) |\leq C.\]
Hence
\[|A_2 (z)| \leq C \int_{\mathcal{I}(\Omega,{2\delta})} |f(U)(y,0)|dy, \quad z \in E_{\delta}.\]
Using this we find
\begin{equation}\label{eq-appendix-G-2}
\int_{E_{\delta}}t^{1-2s} |A_2 (z)|^2 dz \leq C \left(\int_{\mathcal{I}(\Omega,{2\delta})} |f(U)(y,0)|dy\right)^2.
\end{equation}
We have obtained the desired bound of $\int_{E_{\delta}} t^{1-2s} |\nabla U|^2 dz$ through the estimates \eqref{eq-appendix-H}, \eqref{eq-appendix-G-1} and \eqref{eq-appendix-G-2}. The proof is complete.
\end{proof}

\section{The proof of Theorem \ref{letn2}}
%In this section we establish the uniform bound $ \int_{\Omega} u^{p+1} dx \leq C(p,\Omega)$ for $u$ satisfying the equation~\ref{a12u}.
In this section, we prove Theorem \ref{letn2}. Here we set the condition
\\
\noindent \textbf{Condition A}:
\begin{eqnarray*}
\liminf_{u \rightarrow \infty} \frac{f(u)}{u} > \lambda_1^{1/2}, \quad \lim_{u \rightarrow \infty} \frac{f(u)}{u^{{(n+1)}/{(n-1)}}} = 0,
\end{eqnarray*}
with one of the following assumptions
\begin{enumerate}
\item $\Omega$ is convex and
\begin{eqnarray}\label{limsu}
\limsup_{n \rightarrow \infty} \frac{ u f(u) -\theta F(u)}{u^2 f(u)^{2/n}} \leq 0, \quad \textrm{ for some}~ \theta \in [0,\frac{2n}{n-2s}).
\end{eqnarray}
\item Condition \eqref{limsu} holds and the function $u \rightarrow f(u) u^{-\frac{n+2s}{n-2s}}$ is nonincreasing on $(0,\infty)$.
\end{enumerate}
 First we obtain a uniform $L^1$ bound away from the boundary and a uniform $L^{\infty}$ bound near the boundary for positive solutions to \eqref{eq-one}. For $r>0$ we let
\begin{eqnarray*}
\mathcal{O}(\Omega,r) = \{ x  \in \Omega \mid \textrm{dist}(x,\partial \Omega) \leq r \}.
\end{eqnarray*}
Then we have the following result.
\begin{lem}\label{uxc}Let $u$ be a $C^2 (\bar{\Omega})$ solution of \eqref{eq-one} with $f$ satsifying
\begin{eqnarray}\label{nfunl}
\liminf_{n \rightarrow \infty} \frac{f(u)}{u} > \lambda_1^{s}.
\end{eqnarray}
%Suppose that $\Omega$ is strictly convex. 
For each $r>0$, there exists a number $C=C(r,\Omega)>0$ such that
\begin{eqnarray}\label{upxc}
\int_{\mathcal{I}(\Omega,r)} f(u) dx \leq C,
\end{eqnarray}and
\begin{eqnarray}\label{xsux}
\sup_{x \in \mathcal{O}(\Omega,r) } u(x) \leq C.
\end{eqnarray}
\end{lem}
\begin{proof}
Recall that $\phi_1$ is the eigenfunction of $-\Delta\mid_{\Omega}$ with the smallest eigenvalue $\lambda_1 >0$. Using this and \eqref{eq-one}, we get
\begin{equation}\label{11ud}
\begin{split}
\int_{\Omega} \lambda_1^{s} \phi_1 u (x) dx =&\int_{\Omega} (\mathcal{A}_{s} \phi_1) u (x) dx
\\
=& \int_{\Omega} \phi_1 \mathcal{A}_{s} u (x)dx
\\
=& \int_{\Omega} \phi_1 f(u)(x) dx.
\end{split}
\end{equation}
By the condtion \eqref{nfunl} there are constants $\delta>0$ and $C>0$ such that $f(u) > (\lambda_1^{s} + \delta) u - C$ for all $u>0$. With this, \eqref{11ud} gives
\begin{eqnarray*}
\int_{\Omega} \lambda_1^s \phi_1 u dx > \int_{\Omega} (\lambda_1^s +\delta) u \phi_1 dx - \int_{\Omega} C \phi_1 dx,
\end{eqnarray*}
which yields
\begin{eqnarray}\label{eq-l1}
\int_{\Omega} \phi_1 u dx \leq \frac{1}{\delta} \int_{\Omega} C \phi_1 dx  \leq C(\delta, \Omega,f).
\end{eqnarray}
On the other hand, it is well-known $\phi_1 \geq C$ on $\mathcal{I}(\Omega,r)$ with a constant $C=C(r)$. Combining this and \eqref{eq-l1} we get
\begin{eqnarray}\label{phi1u}
\int_{\Omega \setminus \Omega_d} u dx \leq C \int_{\Omega} \phi_1 u dx \leq C.
\end{eqnarray}
From the identity \eqref{11ud}, we obtain the estimate \eqref{upxc}.
\

When $\Omega$ is strictly convex, the moving plane method in \cite{CT, T1} yields that the solution increases along an arbitrary line  toward inside of $\Omega$ starting from any point on $\partial \Omega$. Given this fact, it is well-known  that the estimate \eqref{phi1u} gives the uniform bound near the boundary (see e.g. \cite[Lemma 13.2]{QS}). 
\

For the general domain without the convexity assumption, we make use of the Kelvin transform of $v$ in the space $\mathbb{R}^{n+1}$. Since $\Omega$ is smooth, for a point $x_0$ we can find a ball which contact $x_0$ from the exterior of $\Omega$. We may assume $x_0 =1$ and the ball is $B(0,1)$ without loss of generality. Set
\begin{eqnarray*}
w (z) = |z|^{2s-n} v \left( \frac{z}{|z|^2}\right).
\end{eqnarray*}
Then, $w$ satisfies
\begin{eqnarray*}
\left\{\begin{array}{ll}
\textrm{div}(t^{1-2s} \nabla w) = 0 &\quad \textrm{in}~ \kappa(\mathcal{C}),
\\
w>0&\quad \textrm{in}~ \kappa(\mathcal{C}),
\\
w =0 & \quad \textrm{on} ~\kappa(\partial \Omega \times [0,\infty)),
\\
\partial_{\nu}^s w = g(y,w) &\quad \textrm{on}~ \kappa (\Omega \times\{0\}),
\end{array}\right.
\end{eqnarray*}
where $g(y,w):= f(|y|^{n-2s} w)/ |y|^{n+2s}$. For $\lambda >0$ we set
\begin{itemize}
\item $D_{\lambda} = \kappa(\mathcal{C}) \cap \{ z \in \mathbb{R}^{n+1}_{+}: |z| \leq 1,~ z_1 > 1-\lambda\}$,
\item $\tilde{\partial}D_{\lambda} = D_{\lambda} \cap \partial \mathbb{R}_{+}^{n+1}$,
\item $T_{\lambda}(y)= ( 2-2\lambda - y_1, y_2,\cdots,y_{n+1})$.
\end{itemize}
Let $w_{\lambda} (y) = w(T_{\lambda}(y))$ and $\zeta_{\lambda}=w_{\lambda}-w$ defined on $D_{\lambda}$. We claim that $v_{\lambda} \geq 0$ if $\lambda >0$ is sufficiently small. Set $v_{\lambda}^{-} = \max \{0, - v_{\lambda}\}$. Then,
\begin{equation}
\begin{split}
 0 &= \int_{D_{\lambda}} \zeta_{\lambda}^{-} \textrm{div}(t^{1-2s}\nabla \zeta_{\lambda} ) dx dy
\\
&= \int_{\tilde{\partial}D_{\lambda}} \zeta_{\lambda}^{-} \partial_{\nu}^s \zeta_{\lambda} dx + \int_{D_{\lambda}}t^{1-2s} |\nabla \zeta_{\lambda}^{-}|^2 dx dy.
\end{split}
\end{equation}
We have
\begin{equation}
\begin{split}
\int_{\tilde{\partial}D_{\lambda}} (-\zeta_{\lambda}^{-})\partial_{\nu}^s \zeta_{\lambda} dx &= \int_{\tilde{\partial}D_{\lambda}}(-\zeta_{\lambda}^{-}) ( g(T_{\lambda}x, w_{\lambda}) - g(x,w)) dx
\\
& = \int_{\tilde{\partial}D_{\lambda} \cap \{w_{\lambda} \leq w\}} (w -w_{\lambda}) ( g(x,w) -g(T_{\lambda} x, w_{\lambda})) dx
\\
\end{split}
\end{equation}
Since $u \rightarrow f(u) u^{-\frac{n+2s}{n-2s}}$ is nonincreasing, we see that $g(x,w) \leq g(T_{\lambda} x, w_{\lambda})$ because $|x| \geq |T_{\lambda}(x)|$. Using this we deduce that
\begin{equation}
\begin{split}
\int_{D_{\lambda}} t^{1-2s} |\nabla \zeta_{\lambda}^{-}|^2 dx dy &\leq \int_{\tilde{\partial}D_{\lambda} \cap \{w_{\lambda} \leq w\}} (w-w_{\lambda})(g(x,w)-g(x,w_{\lambda}))dx
\\
&\leq \int_{\tilde{\partial}D_{\lambda} \cap \{w_{\lambda} \leq w\}} (w-w_{\lambda})^2 h(x,w,w_{\lambda}) dx
\\
&= \int_{\tilde{\partial}D_{\lambda} \cap \{w_{\lambda} \leq w\}} (\zeta_{\lambda}^{-})^2 h(x,w,w_{\lambda}) dx,
\end{split}
\end{equation}
where $h(x,w,w_{\lambda})= \frac{g(x,w)-g(x,w_{\lambda})}{w-w_{\lambda}}$. Since $f$ is locally Lipschitz it is bounded by  $\sup_{\tilde{\partial}D_{\lambda}} [|w|+|w_{\lambda}|]$.
By H\"older's inequality we deduce that
\begin{equation}\label{dlzl2}
\begin{split}
\int_{D_{\lambda}} t^{1-2s} |\nabla \zeta_{\lambda}^{-}|^2 dx dy &\leq C \int_{\tilde{\partial}D_{\lambda} \cap \{w_{\lambda} \leq w\}} (\zeta_{\lambda}^{-})^2 dx
\\
&\leq C |\tilde{\partial}D_{\lambda} \cap \{w_{\lambda} \leq w\}|^{2s/n} \| \zeta_{\lambda}^{-} (\cdot,0)\|_{L^{2n/(n-2s)}(\Omega)}^2.
\end{split}
\end{equation}
Using the trace inequality, we get
\begin{eqnarray*}
\| \zeta_{\lambda}^{-} (\cdot,0)\|_{L^{2n/(n-2s)}(\Omega)} \leq C |\tilde{\partial}D_{\lambda} \cap \{w_{\lambda} \leq w\}|^{2s/n} \| \zeta_{\lambda}^{-} (\cdot,0)\|_{L^{2n/(n-2s)}(\Omega)}^2,
\end{eqnarray*}
which yields that $\zeta_{\lambda}^{-} \equiv 0$ for small $\lambda$.
\

Now we set
\begin{eqnarray*}
\eta = \sup \{ \lambda>0 : T_{\lambda} (D_{\lambda}) \subset \kappa (\mathcal{C})\},
\end{eqnarray*}
and
\begin{eqnarray*}
S:= \left\{ 0 <\lambda \leq \frac{\eta}{2} : \zeta_{\lambda} \geq 0 \quad \textrm{on}~D_{\lambda}\right\} \cup \{0\}.
\end{eqnarray*}
We shall prove that $S=[0,\eta/2]$. Since $\zeta_{\lambda}$ is a continuous function of $\lambda$, the set $S$ is closed. Thus, it is enough to show that $S$ is also open in $[0,\eta/2]$. Note that the constant $C$ in the inequality \eqref{dlzl2} can be chosen uniformly for $\lambda \in [0,\eta/2]$ since $\sup_{0<\lambda <\eta/2} \sup_{\tilde{\partial}D_{\lambda}} [|w|+|w_{\lambda}|]$ is bounded.
\

Choose any $0< \lambda_0 < \eta/2$ contained in $S$. Then we have $\zeta_{\lambda_0} \geq 0$. Since $\zeta_{\lambda_0} >0$ on $\kappa (\partial \Omega \times [0,\infty)) \cap D_{\lambda_0}$ and $\textrm{div}(t^{1-2s} \nabla \zeta_{\lambda_0} )\equiv 0$ in $D_{\lambda_0}$, we see that $\zeta_{\lambda_0} >0$ in $D_{\lambda_0}$ by the maximum principle (see e.g. \cite{CS2}). Thus we can find $c>0$ such that
\begin{eqnarray*}
|D_{\lambda_0,c} := \{x \in D_{\lambda_0} : \zeta_{\lambda_0} >c\}| \geq |D_{\lambda_0}| - \delta/2.
\end{eqnarray*}
By continuity, there is $\epsilon >0$ such that $\zeta_{\lambda} > \frac{c}{2}$ and $|D_{\lambda} \setminus D_{\lambda_0}| < \frac{\delta}{2}$ for $\lambda \in [\lambda_0, \lambda_0 + \epsilon)$. For such $\lambda$ we then see that
\begin{eqnarray*}
|\left\{ x \in D_{\lambda}: \zeta_{\lambda} > \frac{c}{2} \right\} | \geq |D_{\lambda} |-\frac{\delta}{2}-\frac{\delta}{2} = |D_{\lambda}| -\delta.
\end{eqnarray*}
This  yields that
\begin{eqnarray*}
| \{ x \in D_{\lambda}: \zeta_{\lambda} \leq 0\}| \leq \delta.
\end{eqnarray*}
Then the inequality \eqref{dlzl2} implies that $\zeta_{\lambda} \geq 0$ for $\lambda \in [\lambda_0,\lambda_0+\epsilon)$. Therefore we have that $w$ increases in any line in $\Omega$ starting from a boundary point. Since $w(x) \geq w(y)$ we deduce that $u(x/|x|^2) \geq c u(y/|y|^2)$ holds with some $c \in (0,1)$ uniformly for $(x,y)$ satisfying $\min (|x|,|y|) > 1/2$. Then we can obtain the $L^{\infty}$ bound near the boundary $\partial \Omega$. It completes the proof.
\end{proof}

\begin{prop}\label{1pn1n} Suppose that $1 < p < \frac{n+2s}{n-2s}$ and let $u \in C^2 (\bar{\mathcal{C}})$ be a solution of the equation \eqref{eq-one} with $f(u)=u^p$. Then there exists a constant $C= C(p,\Omega) >0$ such that
\begin{eqnarray*}
\int_{\Omega} u^{p+1} (x) dx \leq C.
\end{eqnarray*}
Moreover, in the general case  of Theorem \ref{eq-one}, there exists a constant $C=C(f,\Omega) >0$ such that
\begin{eqnarray*}
\int_{\Omega\times \{0\}} \left\{ n F(v) -\frac{n-2s}{2} v f(v) \right\} dx \leq C,
\end{eqnarray*}
where $F(v):= \int_0^{v} f(s) ds$.
\end{prop}
%We are now ready to establish the bound for the $L^{p+1}$-norm of the solution $u$.
\begin{proof}
In this proof, we assume that $\Omega$ is strictly convex. The general case will be proved in the last part of the proof of \mbox{Theorem \ref{letn2}.}
\

We make use of \mbox{Lemma \ref{uxc}} to get a number $\delta >0$ and a constant $C=C(\delta,\Omega)>0$ so that
\begin{eqnarray}\label{xsuxcc}
\sup_{\mathcal{O}(\Omega,\delta)} u(x) \leq C,
\end{eqnarray}
and
\begin{eqnarray}\label{supc}
\int_{\mathcal{I}(\Omega,\delta)} f(u) (x) dx \leq C.
\end{eqnarray}
We apply these estimates to the inequality \eqref{eq_local_poho} . Then we obtain
\begin{equation}
\min_{r \in [\delta,2\delta]} \left| \int_{\Omega \setminus \Omega_{r/2}} n F(U) - \left(\frac{n-2s}{2}\right)  U f(U) dx\right| \leq C.
\end{equation}
The proof is completed.
\end{proof}

\begin{rem}
In the local problem $-\Delta u = u^p$ in $\Omega$, $u=0$ on $\partial \Omega,$ with $1<p< \frac{n+2}{n-2}$, given the $L^{\infty}$ bound \eqref{xsux} of a solution $u $near the boundary, one can use $W^{1,p}$ regularity estimate on $\mathcal{O}(\Omega,\delta)$  to get the $L^{\infty}$ estimates of $\left| \nabla u \right|$ on the  $\mathcal{O}(\Omega,{\delta/2})$. Then, for $f(u)=u^p$ and $p < \frac{n+2}{n-2}$, the Pohozaev identity
\begin{eqnarray*}
\int_{\partial \Omega} \left|\frac{\partial u }{\partial n}\right|^2 (x,\nu) d\sigma = \left( \frac{n}{p+1} - \frac{n-2}{2}\right) \int_{\Omega} u^{p+1} dx
\end{eqnarray*}
gives a uniform bound of $\int_{\Omega} u^{p+1} dx$. Then using the Sobolev embeddings iteratively we can get the uniform bound of $\|u\|_{L^{\infty}(\Omega)}$. This is not applicable to our problem \eqref{eq-one} because the Pohozaev identity is given on the extended domain $\Omega \times [0,\infty)$ as follows (see \cite[Lemma 3.1]{T})
\begin{eqnarray}\label{12lc}
\frac{1}{2} \int_{\partial_L \mathcal{C}} t^{1-2s} |\nabla U|^2 (z, \nu) d \sigma = \left(\frac{n}{p+1} - \frac{n-2s}{2}\right) \int_{\Omega \times \{0\}} |U|^{p+1} dx,
\end{eqnarray}
where $U$ is the harmonic extention of $u$. In this case the left-hand side would not be bounded by using only the $L^{\infty}$ estimate of $u(x)= U(x,0)$ near $\partial {\Omega}$ since the harmonic extension $U(z)$ is made of all values of $u(x)$ for $x \in \Omega$. This is the reason that we rely on the estimates of Proposition \ref{prop-sub-estimate}
\end{rem}

We are now in a position to prove our main theorem.
\begin{proof}[Proof of Theorem \ref{letn2}] For the sake of simplicity, first we prove the theorem for $f(u)=u^p$.
Since $p < \frac{n+2s}{n-2s}$ we get $ q_1 >p$ for $ \frac{p}{p+1} - \frac{1}{q_1} = \frac{2s}{n} -\epsilon$ with sufficiently small $\epsilon >0$. Using Lemma \ref{lem-trace} we get
\begin{eqnarray*}
\| u\|_{q_1} \leq C \| \mathcal{A}_s u \|_{\frac{p+1}{p}} \leq C  \|u^p \|_{\frac{p+1}{p}} \leq C.
\end{eqnarray*}
For $k \geq 1$, we define $q_k$ by the relation $\frac{p}{q_k} - \frac{1}{q_{k+1}} = \frac{2s}{n} -\epsilon$ and stop the sequence when we have $\frac{p}{q_N} < \frac{2s}{n} - \epsilon$. Then, using Lemma \ref{lem-trace}, for $k=1,\cdots, N-1$, we have
\begin{eqnarray*}
\| u\|_{q_{k+1}} \leq C \| \mathcal{A}_s u \|_{\frac{q_k}{p}} \leq C  \|u^{p}\|_{\frac{q_k}{p}} \leq C  1.
\end{eqnarray*}
We then have $\| u\|_{q_N} \leq C 1$, and use Lemma \ref{lem-trace} again to  deduce  $\| u\|_{L^{\infty}}\leq C 1$. It completes the proof when $\Omega$ is convex and $f(u) = u^p$, $p< \frac{n+2s}{n-2s}$.
\

Now we shall prove the theorem for general function $f$ satisfying \textbf{Condition A}. We first see from Proposition \ref{1pn1n} that
\begin{eqnarray}\label{fvdxc}
\int_{\Omega \times \{0\}} \left\{ n F(v) -\frac{n-2s}{2} v f(v) \right\} dx \leq C.
\end{eqnarray}
From the condition \eqref{limsu}, for any $\epsilon>0$, we can find $C_{\epsilon}>0$ such that
\begin{eqnarray}\label{fukce}
uf(u) \leq \theta F(u) + \epsilon u^2 f(u)^{2s/n} + C_{\epsilon}.
\end{eqnarray}
In what follows, $C_{\epsilon}$ may be chosen differently in each line. Using H\"older's inequality and the Sobolev embedding we deduce that
\begin{eqnarray}\label{12u22}
\int_{\Omega} u^2 |f(u)|^{\frac{2s}{n}} dx \leq \| u\|_{\frac{2n}{n-2s}(\Omega)}^{2} \| f(u)\|_{L^1(\Omega)}^{2s/n} \leq C\| \mathcal{A}_{s}^{1/2} u \|_2^2,
\end{eqnarray}
and we have
\begin{equation}\label{12u222}
\begin{split}
\int_{\Omega} u f(u) dx &= \int_{\Omega} u \mathcal{A}_{s} u dx =\int_{\Omega} \mathcal{A}_{s}^{1/2}u \cdot \mathcal{A}_{s}^{1/2} u dx
= \| \mathcal{A}_{s}^{1/2} u\|_{2}^{2}.
\end{split}
\end{equation}
From \eqref{fvdxc} and \eqref{fukce}, we can deduce that
\begin{eqnarray*}
\left( \frac{n}{\theta} -\frac{n-2s}{2}\right) \int_{\Omega} u f(u) dx \leq \frac{\epsilon}{\theta} \int_{\Omega} u^2 f(u)^{\frac{2s}{n}} dx + C_{\epsilon}.
\end{eqnarray*}
Choose $\epsilon =\epsilon (\theta, n) >0$ small enough so that $\left( \frac{n}{\theta} -\frac{n-2s}{2}\right) > \frac{\epsilon}{\theta}$. Then combining  \eqref{12u22} and \eqref{12u222} with the above inequality yields for a constant $C=C(\theta,n)>0$ we have
\begin{equation}
\left(\frac{n}{\theta} -\frac{n-2s}{2}\right)\|\mathcal{A}_s^{1/2} u\|_{2}^2 \leq C\frac{\epsilon}{\theta} \|A_s^{1/2} u\|_2^2 + C_{\epsilon},
\end{equation}
which implies
\begin{eqnarray}\label{eq-A-C}
\| \mathcal{A}_{s}^{1/2} u \|_{2}^{2} \leq C.
\end{eqnarray}
Let $p>1$ and $q = (p+1)\frac{n}{n-2s}$. Then
\begin{equation}\label{pdxce}
\begin{split}
\left( \int_{\Omega} u^{q} dx \right)^{\frac{n-2s}{n}} &= \|u^{(p+1)/2}\|_{\frac{2n}{n-2s}}^2
\\
& \leq C \int_{\Omega \times (0,\infty)} |\nabla u^{\frac{(p+1)}{2}}|^2 dx
=C_p \int_{\Omega \times (0,\infty)} \nabla u \cdot \nabla (u^p) dx
\\
&= C_p \int_{\Omega} \frac{\partial u}{\partial \nu} \cdot u^p dx
\\
& \leq \epsilon C_p \int_{\Omega} u^{\frac{n+2s}{n-2s}} u^p dx + C_{\epsilon}.
\end{split}
\end{equation}
Since $p+1 = \frac{n-2s}{n} q$ we have
\begin{equation}
\begin{split}
\int_{\Omega} u^{\frac{n+2s}{n-2s} } u^p dx &= \int_{\Omega} u^{q (n-2s)/n} u^{\frac{2}{n-2s}} dx
\\
&\leq \left( \int_{\Omega} u^{q(n-2s)/n \cdot \frac{n}{n-2s}} dx \right)^{\frac{n-2s}{n}} \left( \int_{\Omega} u^{\frac{2}{n-2s}\cdot n} dx\right)^{\frac{2s}{n}}
\\
&\leq C\left( \int_{\Omega} u^{q(n-2s)/n \cdot \frac{n}{n-2s}} dx \right)^{\frac{n-2s}{n}} \|\mathcal{A}_{s}^{1/2} u\|_{2}^{\frac{2}{n-2s}}
\\
& \leq C\left( \int_{\Omega} u^{q} dx \right)^{\frac{n-2s}{n}},
\end{split}
\end{equation}
where we used \eqref{eq-A-C} in the last estimate. Combinig this with  \eqref{pdxce} yields that
\begin{eqnarray*}
\left(\int_{\Omega} u^q dx \right)^{1/q} \leq C_{\epsilon},
\end{eqnarray*}
 Since $p$ is an arbitrary number, we can  use Lemma \ref{lem-trace} to conclude that
\begin{eqnarray*}
\| u\|_{L^{\infty}} \leq C_{\epsilon}.
\end{eqnarray*}
It completes the proof.
\end{proof}

\section{The Lane-Emden system}
In this section, we study the Lane-Emden system involving the square root of the Laplacian
\begin{eqnarray}\label{12uv}
\left\{ \begin{array}{ll} A_{s} u = v^p & \quad \textrm{in} ~\mathcal{C},
\\
A_{s} v = u^{q}& \quad \textrm{in} ~\mathcal{C},
\\
u >0,~ v>0 &\quad \textrm{in}~ \mathcal{C},
\\
u=v =0& \quad \textrm{on}~ \partial\mathcal{C}.
\end{array}\right.
\end{eqnarray}
We shall also denote by $u$ and $v$  the harmonic extensions of $u$ and $v$. Then, we have
\begin{eqnarray}\label{eq-ext-sys}
\left\{ \begin{array}{ll} \textrm{div}(t^{1-2s}\nabla u) =\textrm{div}(t^{1-2s}\nabla v) =0 &\quad \textrm{in} ~\mathcal{C},
\\
u =v= 0 & \quad \textrm{on} ~\partial_L \mathcal{C},
\\
\partial_{\nu}^s u = v^{p}, ~\partial_{\nu}^s v = u^{q} &\quad \textrm{on}~ \Omega \times\{0\},
\\
u>0,~ v>0 &\quad \textrm{in} ~\mathcal{C}.
\end{array}\right.
\end{eqnarray}
First, the existence of weak solution and Brezis-Kato type estimate will follow from the same proof of \cite{HV}. We shall obtain a Pohozaev type identity, which proves nonexistence of nontrivial solutions for the system \eqref{12uv} in critical and supercritical cases. Next, we shall establish a moving plane argument. Then, we shall obtain the \emph{a priori} estimate for subcritical cases by applying the framework which was used in the proof of Theorem \ref{letn2}.
\

The existence result follows by applying the proof of \cite[Theorem 1]{HV}. with minor modifications. The proof uses the following result of Benci-Rabinowitz \cite{BR}.
\begin{thm}[Indefinite Functional Theorem]\label{thm-BR} Let $H$ be a real Hilbert sapce with $H = H_1 \oplus H_2$. Suppose $\mathcal{L} \in C^1 (H,\mathbb{R})$ satisfies the Palais-smale condition, and
\begin{enumerate}
\item $\mathcal{L}(u) =\frac{1}{2}(Lu, u)_{H} - \mathcal{H}(u),$ where $L:H\rightarrow H$ is bounded and self-adjoint, and $L$ leaves $H_1$ and $H_2$ invariant;
\item $\mathcal{H}'$ is compact;
\item there exists a subspace $\bar{H} \subset H$ and sets $S \subset H$, $Q \subset \bar{H}$ and constants $\alpha >\omega$ such that
\begin{enumerate}
\item $S \subset H_1$ and $\mathcal{L}\mid_{S} \geq \alpha,$
\item $Q$ is bounded and $\mathcal{L} \leq \omega$ on the boundary $\partial Q$ of $Q$ in $\bar{H}$,
\item $S$ and $\partial Q$ link.
\end{enumerate}
\end{enumerate}
Then $\mathcal{L}$ possesses a critical value $c \geq \alpha$.
\end{thm}
We set
\begin{itemize}\item $E^a (\Omega) = H^a (\Omega) \times H^{2s-a} (\Omega),$ \quad $0<a<2s$.
\item $E^{\pm} = \{ (u, \pm(-\Delta)^{a-2s} u): u \in H^a (\Omega)\}$.
\end{itemize}
We then have
\begin{equation}
E^a (\Omega) = E^{+} \oplus E^{-} = \{ u=u^{+}+ u^{-}, ~u^{\pm} \in E^{\pm}\}.
\end{equation}
We easily see that $E^{\pm}$ have their orthonormal basis
\begin{equation}
\left\{ \frac{1}{\sqrt{2}} (\lambda_k^{-a/2}\phi_k, \pm \lambda_k^{a/2 -1}\phi_k ) : k=1,2,\cdots \right\}.
\end{equation}
Let
\begin{equation}
L =\left( \begin{array}{ll} 0 & (-\Delta)^{2s-a} \\ (-\Delta)^{a-2s} &0 \end{array}\right).
\end{equation}
Then,
\begin{equation}
A(u) =\frac{1}{2} \langle (-\Delta)^s u, u \rangle = \frac{1}{2} (Lu, u)_{E^r}.
\end{equation}

\begin{proof}[Proof of Theorem \ref{thm-ex}]
We apply Theorem \ref{thm-BR} with the spaces $H= E^{r}(\Omega)$, $H_1 = E^{+}$, and $H_2 = E^{-}$. In this setting, it follows from the proof of \cite[Theorem1]{HV} with minor changes that the conditions (1)-(3) of Theorem \ref{thm-BR} are satisfied. Then the existence of a weak solution $(u,v)$ follows. We omit the detail for the simplicity of exposition. The difference of the range of $(p,q)$ is due to the different ranges of the Sobolev inequalities.
\end{proof}
%\begin{thm} Assume $(p,q)$ is sub-critical or critical.
%Let $(u,v)$ be a weak  solution of the system \eqref{12uv}. Then, $u \in C^{2,\alpha}$ and $v\in C^{2,\alpha}$ for some $0< \alpha<1$.
%\end{thm}
We have the following Brezis-Kato type result.
\begin{prop}Assume that $(p,q)$ is critical or sub-critical. Let $(u,v)$ be a weak solution of \eqref{eq-sys}. Then we have $u  \in L^{\infty}(\Omega)$ and $v \in L^{\infty}(\Omega)$. 
\end{prop}
\begin{proof}
Note that
\begin{equation}
\left\{\begin{array}{ll} \mathcal{A}_s v = a(x) u &\quad \textrm{in}~\Omega,
\\
\mathcal{A}_s u = b(x) v&\quad \textrm{in}~\Omega.
\end{array}
\right.
\end{equation}
Since $a(x) \in L^{\frac{p+1}{p-1}} (\Omega)$ we have
\begin{equation}
a(x)u(x) = q_{\epsilon}(x) u(x) + f_{\epsilon}(x),
\end{equation}
where $ f_{\epsilon} \in L^{\infty}(\Omega)$ and $\|q_{\epsilon}\|_{\frac{p+1}{p-1}(\Omega)} < \epsilon.$ We have
\begin{equation} 
u(x) = (\mathcal{A}_s)^{-1}(b v)(x).
\end{equation}
Hence,
\begin{equation}
v = (\mathcal{A}_s)^{-1} \left[ q_{\epsilon} (\mathcal{A}_s)^{-1}(b v)\right] + (\mathcal{A}_s)^{-1} f_{\epsilon}(x).
\end{equation}
From Lemma \ref{lem-trace} and H\"older's inequality, we have the following embedding properties of linear maps; Fix $\alpha>1$, then
\begin{itemize}
\item $w \rightarrow b(x)w$ is bounded form $L^{\alpha}(\Omega)$ to $L^{\beta}(\Omega)$ for
\begin{equation}
\frac{1}{\beta} = \frac{q-1}{q+1} +\frac{1}{\alpha}.
\end{equation}
\item $w \rightarrow (\mathcal{A}_s)^{-1} w$ is bounded from $L^{\beta}(\Omega)$ to $L^{\gamma}(\Omega)$ for 
\begin{equation}
2s= n \left(\frac{1}{\beta}-\frac{1}{\gamma}\right).
\end{equation}
\item $w \rightarrow q_{\epsilon}(x) w$ is bounded from $L^{\gamma}$ to $L^a$ with the norm $\| q_{\epsilon}\|_{L^{\frac{p+1}{p-1}}(\Omega)}$ for
\begin{equation}
\frac{1}{a} = \frac{p-1}{p+1} + \frac{1}{\gamma}.
\end{equation}
\item $w \rightarrow (\mathcal{A}_s)^{-1} w$ is bounded from $L^{a}(\Omega)$ to $L^{b}(\Omega)$ for
\begin{equation}
2s = n \left(\frac{1}{a}-\frac{1}{b}\right).
\end{equation}
\end{itemize}
Combining these facts, we see that  the map $w \rightarrow (\mathcal{A}_s)^{-1} \left[ q_{\epsilon} (\mathcal{A}_s)^{-1}(b w)\right]$ is bounded from $L^{\alpha}(\Omega)$ to $L^{\alpha}(\Omega)$ for any $\alpha >1$. Thus, 
\begin{equation}\label{eq-bre}
\begin{split}
\| v\|_{L^{\alpha}} &\leq \| (-\Delta)^{-1}\left[q_{\epsilon} (\mathcal{A}_s)^{-1} (bv)\right]\|_{L^{\alpha}(\Omega)} + \|(\mathcal{A}_s)^{-1} f_{\epsilon}\|_{L^{\alpha}(\Omega)}
\\
&\leq C\|q_{\epsilon}\| \|v\|_{L^{\alpha}(\Omega)} +  \|(\mathcal{A}_s)^{-1} f_{\epsilon}\|_{L^{\alpha}(\Omega)}.
\end{split}
\end{equation}
Since $\|q_{\epsilon} \|_{\frac{p+1}{p-1}} \leq \epsilon$, we can deduce from \eqref{eq-bre} that $\|v\|_{L^{\alpha}} \leq C$ for some $C>0$. Then using Lemma \ref{lem-trace} we  deduce that $u \in L^{\infty}(\Omega)$. From this we also get $v \in L^{\infty}(\Omega)$. The lemma is proved
\end{proof}
Now we recall the regularity result form \cite{CS, T1}. Consider weak solution $U \in H_{0,L}^s (\mathcal{C}) \cap L^{\infty}(\mathcal{C})$ to the problem
\begin{equation}
\left\{ \begin{array}{ll}
\textrm{div}(t^{1-2s} \nabla U) =0&\quad \textrm{in}~\mathcal{C},
\\
U=0&\quad \textrm{on}~\partial_L \mathcal{C},
\\
\partial_{\nu}^s U (x,0) = g(x) &\quad \textrm{on}~\Omega \times \{0\}.
\end{array}
\right.
\end{equation}
Then, for $g \in C^{\alpha}(\Omega)$, we have
\begin{equation}
\left\{ \begin{array}{ll} 
v \in C^{\alpha+2s} (\Omega)&\quad \textrm{if} ~\alpha +2s <1,
\\
v \in C^{1,\alpha+2s-1} (\Omega)&\quad \textrm{if}~ \alpha +2s  \geq 1.
\end{array}
\right.
\end{equation}
Using this result iteratively, we can prove the following result.
\begin{prop}
Let $(u,v)$ is a weak solution of \eqref{eq-sys} such that $u \in H^{s_1}(\Omega) \cap L^{\infty}(\Omega)$ and $v \in H^{s_2}(\Omega) \cap L^{\infty}(\Omega)$ for some $s_1>0$ and $s_2 >0$. Then it holds that $u \in C^{1,\alpha}(\bar{\Omega})$ and $v \in C^{1,\alpha} (\bar{\Omega})$ for any $\alpha \in (0,1)$. 
\end{prop}

We shall obtain a Pohozaev type identity for the system \eqref{eq-ext-sys}. It will gives the nonexistence result for the critical and supercritical cases.
\begin{thm}
Suppose that $(u,v) \in C^{2} (\bar{\mathcal{C}}) \times C^{2} (\bar{\mathcal{C}})$ satisfies
\begin{eqnarray}\label{uvoi}
\left\{\begin{array}{ll}
\textrm{div}(t^{1-2s}) u = \textrm{div}(t^{1-2s}\nabla v) = 0 &\quad \textrm{in}~ \mathcal{C},
\\
u=v=0&\quad \textrm{on} ~\partial_L \mathcal{C}.
\end{array}
\right.
\end{eqnarray}
Then we have
\begin{equation}\label{lcxnu}
\begin{split}
& \int_{\partial_{L} \mathcal{C}} t^{1-2s} (z \cdot \nu) \frac{\partial u}{\partial \nu} \frac{\partial v}{\partial \nu} d\sigma
\\
&\quad\quad = -\int_{\Omega \times \{y=0\}} [ (x, \nabla_x v)\partial_{\nu}^s u + (x, \nabla_x u) \partial_{\nu}^s v ] dx - (n-2s) \int_{\mathcal{C}} t^{1-2s} \nabla u \cdot \nabla v dx.
\end{split}
\end{equation}
\end{thm}
\begin{proof}
We have
\begin{equation}
\begin{split}
&\textrm{div}[ t^{1-2s} z \cdot \nabla v) \nabla u + t^{1-2s} z\cdot \nabla u \nabla v]
\\
&\quad \quad = (z, \nabla v) \textrm{div}(t^{1-2s}\nabla u) + (z, \nabla u) \textrm{div}(t^{1-2s}\nabla v ) + t^{1-2s} z \cdot \nabla ( \nabla u \cdot \nabla v) + 2 t^{1-2s} \nabla u \cdot \nabla v.
\end{split}
\end{equation}
Therefore, from \eqref{uvoi} in $\mathcal{C}$, we have
\begin{equation}
\textrm{div} [ t^{1-2s}(z,\nabla v) \nabla u + t^{1-2s}(z, \nabla u) \nabla v] = t^{1-2s}z \cdot \nabla ( \nabla u \cdot \nabla v) + 2 t^{1-2s} \nabla u \cdot \nabla v.
\end{equation}
We also have
\begin{eqnarray*}
\textrm{div} [t^{1-2s}  (z) ( \nabla u \cdot \nabla v)] &=& (\textrm{div} t^{1-2s}z) (\nabla u \cdot \nabla v) + t^{1-2s} z \cdot \nabla (\nabla  u \cdot \nabla v)
\\
&=& (n+2-2s)t^{1-2s} (\nabla u \cdot \nabla v) + t^{1-2s} z \cdot \nabla (\nabla u \cdot \nabla v).
\end{eqnarray*}
The above two formulas gives the following equality:
\begin{eqnarray}\label{divx}
\textrm{div} [ t^{1-2s} (z,\nabla v) \nabla u + t^{1-2s} (z, \nabla u) \nabla v] - \textrm{div} (t^{1-2s} z (\nabla u \cdot \nabla v)) + (n-2s) t^{1-2s} \nabla u \cdot \nabla v = 0.
\end{eqnarray}
Using the divergence theorem and the fact that $u=v=0$ on $\partial \Omega \times [0,\infty)$, we get
\begin{equation}
\begin{split}
&\int_{\Omega \times (0,R)} \textrm{div} [t^{1-2s} (z,\nabla v) \nabla u + t^{1-2s} (z, \nabla u) \nabla v]dx
\\
&\quad = \int_{\Omega \times (0,R)} 2 t^{1-2s} (z\cdot \nu) \frac{\partial u}{\partial \nu} \cdot \frac{\partial v}{\partial \nu} d \sigma + \int_{\Omega \times \{y=0\}} [ (x, \nabla_x v) \partial_{\nu}^s u + (x, \nabla_x u) \partial_{\nu}^s v] dx
\\
&\quad \quad+ \int_{\Omega \times \{y=R\}}t^{1-2s} [(x, \nabla_x v)(\nabla u,\nu) + (x, \nabla_x u) ( \nabla v, \nu) ] dx,
\end{split}
\end{equation}
and
\begin{equation}
\int_{\Omega \times (0,R)} \textrm{div} (t^{1-2s}z ( \nabla u \cdot \nabla v)) dx = \int_{\partial \Omega \times (0,R)} t^{1-2s}(z \cdot \nu) ( \frac{\partial u}{\partial \nu} \cdot \frac{\partial v }{\partial \nu} ) d \sigma + \int_{\Omega \times \{ y=R\}} R^{2-2s} (\nabla u \cdot \nabla v) dx.
\end{equation}
Letting $R \rightarrow \infty$ we obtain
\begin{equation}
\begin{split}
&\int_{\mathcal{C}} \textrm{div} [t^{1-2s} (z,\nabla v) \nabla u + t^{1-2s}(z, \nabla u) \nabla v]dx
\\
&\quad = \int_{\mathcal{C}} 2t^{1-2s} (z\cdot \nu) \frac{\partial u}{\partial \nu} \cdot \frac{\partial v}{\partial \nu} d \sigma + \int_{\Omega \times \{y=0\}} [ (x, \nabla_x v) \partial_{\nu}^s u + (x, \nabla_x u) \partial_{\nu}^s v] dx,
\end{split}
\end{equation}
and
\begin{equation}
\int_{\mathcal{C}} \textrm{div} (t^{1-2s}z ( \nabla u \cdot \nabla v)) dx = \int_{\mathcal{C}}t^{1-2s} (z \cdot \nu) ( \frac{\partial u}{\partial \nu} \cdot \frac{\partial v }{\partial \nu} ) d \sigma.
\end{equation}
Integrating \eqref{divx} over $\mathcal{C}$ and using the above two formulas we obtain
\begin{equation*}
 \int_{\mathcal{C}}  t^{1-2s} (z\cdot \nu) \frac{\partial u}{\partial \nu} \cdot \frac{\partial v}{\partial \nu} d \sigma + \int_{\Omega \times \{y=0\}} [ (x, \nabla_x v) \partial_{\nu}^s u + (x, \nabla_x u) \partial_{\nu}^s v] dx
= (n-2s)\int_{\mathcal{C}} t^{1-2s} \nabla u \nabla v dx,
\end{equation*}
which is the desired identity \eqref{lcxnu}.
\end{proof}
\begin{proof}[Proof of Theorem \ref{thm-noex}]
We may assume that $\Omega$ is starshaped with respect to the origin, that is, $(x\cdot \nu) >0$ for any $x \in \partial_{L}\Omega$. It easily implies that $(x\cdot \nu)>0$ holds also for $x \in \partial_{L}\mathcal{C}$.
\

Suppose that $(u,v) \in C^{2}(\bar{\mathcal{C}}) \times C^{2}(\bar{\mathcal{C}})$ satisfies \eqref{eq-ext-sys} and denote also by $u$ and $v$  the harmonic extensions of $u$ and $v$. Let $f(v) = v^p$ and $g(u)=u^p$ and set
\begin{eqnarray*}
F(v) = \int_0^{v} f(s)ds \quad \textrm{and} \quad G(u) = \int^{u}_0 g(s) dx.
\end{eqnarray*}
Because $F(0) =0 $ and $u=0$ on $\partial \Omega \times \{0\}$, we get
\begin{equation}
\begin{split}
&\int_{\Omega \times \{0\}} (x, \nabla_x v) \partial_{\nu}^s u (x) dx = \int_{\Omega \times\{0\}} (x, \nabla_x v) f(v) dx
\\
&\quad = \int_{\Omega \times \{0\}} (x, \nabla_x F(v)) dx = - \int_{\Omega \times \{0\}} n F(v) dx.
\end{split}
\end{equation}
Likewise, we have
\begin{eqnarray*}
\int_{\Omega \times\{0\}} (x,\nabla_x u) \partial_{\nu}^s v(x) dx = -\int_{\Omega \times\{0\}} n G (u) dx.
\end{eqnarray*}
For a solution $(u,v)$ of the system \eqref{12uv}, we get
\begin{equation}\label{2lcx}
\begin{split}
& \frac{1}{2} \int_{\partial_{L}\mathcal{C}}t^{1-2s} (x\cdot \nu) \frac{\partial u}{\partial \nu} \frac{\partial v}{\partial \nu} d\sigma
\\
&\quad =  \int_{\Omega \times \{y=0\}} \biggl( \frac{n}{p+1}- (n-2s) \theta\biggr) v^{p+1} + \biggl(\frac{n}{q+1} - (n-2s) (1-\theta) \biggr) u^{q+1} dx.
\end{split}
\end{equation}
Since $u=v=0$ on $\partial_L \mathcal{C}$ we see that $\frac{\partial u }{\partial \nu} \geq 0$ and $\frac{\partial v}{\partial \nu} \geq 0$ on $\partial_L \mathcal{C}$. If $(p,q)$ is super-critical we can find $\theta \in (0,1)$ such that
\begin{eqnarray*}
\frac{n}{p+1} - (n-2s) \theta <0\quad \textrm{and}\quad \frac{n}{q+1} - (n-2s)(1-\theta) <0.
\end{eqnarray*}
It implies that $u \equiv v \equiv 0$ on $\Omega \times \{0\}$. In the critical case, we have
\begin{eqnarray*}
\frac{1}{2} \int_{\partial_L \mathcal{C}} t^{1-2s} (x,\nu) \frac{\partial u}{\partial \nu}\frac{\partial v}{\partial \nu} d \sigma,
\end{eqnarray*}
which implies that $\frac{\partial u}{\partial \nu}(x_0) =0$ or $\frac{\partial v}{\partial \nu}(x_0)=0$. Since $\textrm{div}(t^{1-2s}\nabla u )= \textrm{div}(t^{1-2s}\nabla v)=0$ and  $u$ and $v$ are nonnegative on $\mathcal{C}$, it follows from Hopf's lemma that $u \equiv 0 $ or $v \equiv 0$, which yields that $u \equiv v \equiv 0$. The proof is complete.
\end{proof}
Next, we shall establish the moving plane argument, which will give the symmetry result and the $L^{\infty}$ bound near the boundary of positive solutions to \eqref{12uv}. As a preliminary step, we need the following lemma.
\begin{lem}\label{thenu}
Assume that $ c\leq 0, ~d \leq 0$ and $\Omega$ is a bounded (not necessary smooth) domain of $\mathbb{R}^n$ and set $\mathcal{C} = \Omega \times (0,\infty)$. Suppose  $u, v \in C^{2} (\bar{\mathcal{C}}) \cap L^{\infty}(\mathcal{C})$ is a solution of the system
\begin{eqnarray}\label{vwin}
\left\{ \begin{array}{ll} \textrm{div}(t^{1-2s}\nabla u )=  \textrm{div}(t^{1-2s}\nabla v) = 0 &\quad \textrm{in} ~\mathcal{C},
\\
u \geq 0,~ v \geq 0 &\quad \textrm{on}~\partial_L \mathcal{C},
\\
\partial_{\nu}^s u + c(x) v \geq 0 &\quad \textrm{on} ~\Omega \times \{0\},
\\
\partial_{\nu}^s v + d(x) u \geq 0 &\quad \textrm{on} ~\Omega \times \{0\},
\end{array}
\right.
\end{eqnarray}
and there is some point $x_0 \in \mathcal{C}$ such that $u(x_0)=v(x_0) =0$. Then, there exists $\delta >0$ depending only on  $\| c \|_{L^{\infty} (\Omega)}$, $\| d \|_{L^{\infty}}$ and $n$ such that if
\begin{eqnarray*}
|\Omega \cap \{ u(\cdot, 0) < 0\}|\cdot |\Omega \cap \{ v(\cdot,0) <0\}| \leq \delta,
\end{eqnarray*}
then $u\geq 0$ and $v \geq 0$ in $\mathcal{C}$.
\end{lem}
\begin{proof}
Set $u^{-} = \max \{0, - u\}$ and $v^{-} = \max \{0,-v\}$. As $u^{-}=v^{-} =0$ on $\partial \Omega \times [0,\infty)$, we get
\begin{eqnarray*}
0= \int_{\mathcal{C}} u^{-}  \textrm{div}(t^{1-2s}\nabla u) dx dy = \int_{\Omega \times\{0\}} u^{-} \partial_{\nu}^s u dx + \int_{\mathcal{C}} t^{1-2s} |\nabla u^{-} |^2 dx dy.
\end{eqnarray*}
Then, as $c\leq 0$, we deduce that
\begin{equation}\label{cu2dx}
\begin{split}
\int_{\mathcal{C}} t^{1-2s} |\nabla u^{-}|^2 dx dy =& - \int_{\Omega \times \{0\}} v^{-} \partial_{\nu}^s u dx
\\
= & \int_{\Omega \times \{0\}} u^{-} c v dx
\\
\leq & \int_{\Omega \times \{0\}} u^{-} (-c) v^{-} dx
\\
\leq & |\Omega \cap \{ u^{-} (\cdot,0) >0 \}|^{2s/n} \|c\|_{L^{\infty}  (\Omega)} \| u^{-}\|_{L^{2n/(n-2s)}(\Omega)} \cdot \| v^{-} \|_{L^{2n/(n-2s)} (\Omega)}.
\end{split}
\end{equation}
By the same argument for $v^{-}$, we get
\begin{eqnarray}\label{cv2dx}
\int_{\mathcal{C}} t^{1-2s} |\nabla v^{-}|^2 dx dy&\leq & |\Omega \cap \{ u^{-} (\cdot,0) >0 \}|^{2s/n} \|d\|_{L^{\infty}  (\Omega)} \| u^{-}\|_{L^{2n/(n-2s)}(\Omega)} \cdot \| v^{-} \|_{L^{2n/(n-2s)} (\Omega)}.
\end{eqnarray}
Multipliying the above two inequalities, we obtain
\begin{equation}
\begin{split}
&\left( \int_{\mathcal{C}} t^{1-2s} |\nabla u^{-}|^2 dx dy \right) \left( \int_{\mathcal{C}} t^{1-2s}|\nabla v^{-} |^2 dx dy \right)
\\
&\quad\quad \leq |\Omega \cap \{ u^{-}(\cdot,0) >0 \}|^{1/n} |\Omega \cap \{ v^{-}(\cdot,0) >0\}|^{2s/n} \| c\|_{L^{\infty}(\Omega)} \| d\|_{L^{\infty}(\Omega)} \|u^{-}\|_{L^{2n/(n-2s)}(\Omega)}^2 \| v^{-}\|_{L^{2n/(n-2s)}(\Omega)}^2.
\end{split}
\end{equation}
We now use the Sobolev trace inequality
\begin{eqnarray*}
S_0  {\| u^{-}(\cdot,0)\|_{L^{2n/(n-2s)}(\Omega)}^2}\leq {\int_{\mathcal{C}} |\nabla u^{-} |^2 dx dy}
\end{eqnarray*}
and
\begin{eqnarray*}
 S_0 {\| v^{-}(\cdot,0)\|_{L^{2n/(n-2s)}(\Omega)}^2} \leq {\int_{\mathcal{C}} |\nabla v^{-} |^2 dx dy}.
\end{eqnarray*}
Then it follows that
\begin{equation*}
\begin{split}
&S_0^2 \| u^{-}(\cdot,0)\|_{L^{2n/(n-2s)}(\Omega)}^2 \| v^{-}(\cdot,0)\|_{L^{2n/(n-2s)}(\Omega)}^2
\\
&\quad\quad \leq |\Omega \cap \{ u^{-}(\cdot,0) >0 \}|^{2s/n} |\Omega \cap \{ v^{-}(\cdot,0) >0\}|^{2s/n} \| c\|_{L^{\infty}(\Omega)} \| d\|_{L^{\infty}(\Omega)} \|u^{-}\|_{L^{2n/(n-2s)}(\Omega)}^2 \| v^{-}\|_{L^{2n/(n-2s)}(\Omega)}^2.
\end{split}
\end{equation*}
If we choose $\delta$ so that $S_0^2 > \delta^{1/n} \| c\|_{L^{\infty}(\Omega)} \| d\|_{L^{\infty}(\Omega)}$, the above inequality yields that $u^{-} \equiv 0$ or $v^{-} \equiv 0$. Say $u^{-} \equiv 0$, then we have $\int_{\mathcal{C}} |\nabla v^{-}|^2 dx dy=0$ from \eqref{cv2dx}. Thus we have $\nabla v^{-} \equiv 0$, and since $v(x_0) =0$, we conclude that $v^{-} \equiv  0$. The proof is complete.
\end{proof}

For $y \in \partial \Omega$ and $\lambda >0$ we set
\begin{eqnarray*}
T(y, \lambda) : = \{ x \in \mathbb{R}^n : \langle y-x, \nu (y) \rangle=\lambda\},
\\
\Sigma( y, \lambda) : = \{ x \in \Omega : \langle y-x, \nu (y) \rangle \leq \lambda\},
\end{eqnarray*}
and define $R(y,\lambda)$ be the reflection with respect to the hyperplane $T(y,\lambda)$. We also set $\Sigma'(y,\lambda) := R(y,\lambda) \Sigma(y,\lambda)$ and
\begin{eqnarray}\label{lysup}
\lambda_y := \sup \{ \lambda >0 : \Sigma(y,\lambda) \subset \Omega\}.
\end{eqnarray}
\begin{lem}\label{uvc2}
Suppose that $(u,v) \in C^2 (\Omega)$ is a solution of \eqref{eq-sys}. Then, for any $y \in \partial \Omega$ and  $x \in \Sigma(y,\lambda)$, we have
\begin{eqnarray*}
u(R(y,\lambda)x ) \geq u(x) \quad \textrm{and}\quad  v (R(y,\lambda) x) \geq v(x)
\end{eqnarray*}
for any  $\lambda \in (0, \lambda_y].$
\end{lem}
\begin{proof} We may assume that $0 \in \partial \Omega$ and $\nu = (1,0)$ is a normal direction to $\partial \Omega$ at this point. It is sufficient to prove the lemma at this point. For $\lambda>0$ we set
\begin{eqnarray*}
\Sigma_{\lambda} = \{ (x_1,x') \in \Omega : x_1 > \lambda\}\quad \textrm{and}\quad T_{\lambda} = \{ (x_1,x') \in \Omega : x_1 = \lambda\}.
\end{eqnarray*}
For $x \in \Sigma_{\lambda}$, define $x_{\lambda} = ( 2 \lambda - x_1, x')$. From the defintion \eqref{lysup} we see
\begin{eqnarray*}
\{ x_{\lambda} : x \in \Sigma_{\lambda}\} \subset \Omega\quad\quad \forall \lambda < \lambda_0.
\end{eqnarray*}
We denote also by $u$ and $v$ the harmonic extension of $u$ and $v$ in $\mathcal{C}$. Then, $(u,v) \in C^2 (\bar{\mathcal{C}})$ satisfies
\begin{eqnarray}\label{dudv0}
\left\{ \begin{array}{ll} \textrm{div}(t^{1-2s}\nabla u) =\textrm{div}(t^{1-2s}\nabla v) =0 &\quad \textrm{in} ~\mathcal{C},
\\
u =v= 0 & \quad \textrm{on} ~\partial_L \mathcal{C},
\\
\partial_{\nu}^s u = v^{p}, ~\partial_{\nu}^s v = u^{q} &\quad \textrm{on}~ \Omega \times\{0\},
\\
u>0,~ v>0 &\quad \textrm{in} ~\mathcal{C}.
\end{array}\right.
\end{eqnarray}
For $(x,y) \in \Sigma_{\lambda} \times [0,\infty),$ we set
\begin{eqnarray*}
u_{\lambda} (x,y) = u(x_{\lambda}, y) = u (2\lambda -x_1 , x', y)
\end{eqnarray*}
and
\begin{eqnarray*}
\alpha_{\lambda} (x,y) = (u_{\lambda} - u) (x,y),\quad \beta_{\lambda}(x,y) = (v_{\lambda} - v) (x,y).
\end{eqnarray*}
Then we have $u_{\lambda}=v_{\lambda}=0$ on $T_{\lambda}\times [0,\infty)$ and obtain from \eqref{dudv0} that $u_{\lambda}>0$ and $v_{\lambda} > 0$ on $(\partial \Omega \cap \bar{\Sigma}_{\lambda}) \times [0,\infty)$. Since $\partial \Sigma_{\lambda} = T_{\lambda}\cup (\partial \Omega \cap \bar{\Sigma}_{\lambda}) $ we see that $(\alpha_{\lambda}, \beta_{\lambda})$ satisfies
\begin{eqnarray*}
\left\{ \begin{array}{ll}
\textrm{div}(t^{1-2s}\nabla \alpha_{\lambda}) = \textrm{div}(t^{1-2s}\nabla \Delta \beta_{\lambda}) =0 &\quad \textrm{in} ~\Sigma_{\lambda} \times (0, \infty),
\\
\alpha_{\lambda} \geq 0, ~\beta_{\lambda} \geq 0 &\quad \textrm{on}~ (\partial \Sigma_{\lambda}) \times (0,\infty),
\\
\partial_{\nu}^s \alpha_{\lambda} + c_{\lambda}(x) \beta_{\lambda} =0 &\quad \textrm{on}~ \Sigma_{\lambda} \times \{0\},
\\
\partial_{\nu}^s \beta_{\lambda} + d_{\lambda}(x) \alpha_{\lambda} =0 &\quad \textrm{on}~ \Sigma_{\lambda} \times \{0\},
\end{array}
\right.
\end{eqnarray*}
where
\begin{eqnarray*}
c_{\lambda}(x,0) = - \frac{v_{\lambda}^{p} - v^{p}}{v_{\lambda} - v}\quad \textrm{and} \quad d_{\lambda}(x,0) = - \frac{u_{\lambda}^{p} - u^{p}} {u_{\lambda} -u}.
\end{eqnarray*}
Note that $c_{\lambda} \leq 0$ and $d_{\lambda} \leq 0$. Now we choose a small number $\kappa >0$ so that the set  $\Sigma_{\lambda}$ has small measure  for $0<\lambda <\kappa$. We then deduce from Lemma \ref{vwin} that, for all $\lambda \in (0,\kappa)$,
\begin{eqnarray*}
\alpha_{\lambda} \geq 0 \qquad \textrm{and} \quad \beta_{\lambda} \geq 0 \quad \textrm{on}~ \Sigma_{\lambda} \times (0,\infty).
\end{eqnarray*}
The strong maximum principle implies that $\alpha_{\lambda} $ and $\beta_{\lambda} $ are identically equal to zero or strictly positive in $\Sigma_{\lambda} \times (0,\infty)$. Since $\lambda >0$, we have $\alpha_{\lambda} >0$ and $\beta_{\lambda} >0$ in $(\partial \Omega \cap \partial \Sigma_{\lambda}) \times (0,\infty)$, and so we deduce that $\alpha_{\lambda} >0$ and $\beta_{\lambda}>0$ in $\Sigma_{\lambda} \times (0,\infty)$.
\

We let $\lambda_1 = \sup\{ \lambda >0 | \alpha_{\lambda} \geq 0 ~\textrm{and} ~\beta_{\lambda} \geq 0 ~ \textrm{in} ~\Sigma_{\lambda} \times (0,\infty)\}.$ We claim that $\lambda_1 =\lambda_0$. With a view to  contradiction, we suppose that $\lambda_1 < \lambda_0$. By continuity we have $\alpha_{\lambda_1} \geq 0$ and $\beta_{\lambda_1} \geq 0$ in $\Sigma_{\lambda_1} \times (0,\infty)$. As before, from the strong maximum principle, we have that $\alpha_{\lambda_1} >0$ and $\beta_{\lambda_1} >0$ in $\Sigma_{\lambda_1} \times (0,\infty)$. Next, let $\delta>0$ be a constant and find a compact set $K \subset \Sigma_{\lambda_1}$ such that $|\Sigma_{\lambda_1}\setminus K|\leq \delta/2$. We have $\alpha_{\lambda_1} \geq \mu >0$ and $\beta_{\lambda_1} \geq \eta>0$ in $K$ for some constant $\eta$, since $K$ is compact. Thus, we obtain that $\alpha_{\lambda_1 + \epsilon}(\cdot, 0) \geq 0$ and $\beta_{\lambda_1 +\epsilon}(\cdot, 0) \geq 0$ in $K$ and that $|\Sigma_{\lambda_1 + \epsilon} \setminus K|\leq \delta$ for sufficiently small $\epsilon >0$.
\

By using Lemma \ref{vwin} in $\Sigma_{\lambda_1 + \epsilon} \times (0,\infty)$ to the function $(\alpha_{\lambda_1 +\epsilon}, \beta_{\lambda_1 + \epsilon})$, we have that $\alpha_{\lambda_1 +\epsilon} \geq 0$ and $\beta_{\lambda_1 +\epsilon} \geq 0$ in $K$. Thus $\{ \alpha_{\lambda_1 +\epsilon} <0\}, \{ \beta_{\lambda_1 +\epsilon} <0\} \subset \Sigma_{\lambda_1 + \epsilon} \setminus K$, which have measure at most $\delta$. We take $\delta$ to be the constant of Lemma \ref{vwin}. Then we deduce that
\begin{eqnarray*}
\alpha_{\lambda_1 +\epsilon} \geq 0 \quad \textrm{and} \quad \beta_{\lambda_1 + \epsilon} \geq 0 \quad \textrm{in} ~\Sigma_{\lambda_1 +\epsilon} \times (0,\infty).
\end{eqnarray*}
This is a contradiction to the definition of $\lambda_1$. Thus, we have that $\lambda_1 =\lambda_0$, which proves the lemma.
\end{proof}
This lemma gives the following symmetry result of Theorem \ref{thm-syme}.
We are now ready to prove Theorem \ref{letn22}.
\begin{proof}[Proof of Theorem \ref{letn22}]
Since $\Omega$ is convex and smooth, there exist constants $\lambda_0 >0$ and $c_0 >0$ such that
\begin{eqnarray}\label{yand}
\Sigma' (y,\lambda) \subset \Omega,\quad \lambda \leq \lambda_0,\quad \textrm{and}\quad (\nu(x),\nu(y)) > c_0, \quad x \in \partial \Sigma (y, \lambda_0) \cap \partial \Omega.
\end{eqnarray}
If $(p,q)$ is sub-critical, we may choose $\theta \in (0,1)$ so that
\begin{eqnarray*}
\frac{n}{p+1} - (n-2s) \theta >0 \quad \textrm{and} \quad \frac{n}{q+1} - (n-2s)(1-\theta) >0.
\end{eqnarray*}
Recall that $\phi_1$ is the positive eigenfunction of $-\Delta$ on $\Omega$ with the eigenvalue $\lambda_1$. Using \eqref{eq-sys} we obtain
\begin{eqnarray*}
\int_{\Omega} v^{p} \phi_1 dx = \int_{\Omega} \sqrt{\lambda_1} u \phi_1 dx\quad \textrm{and}\quad \int_{\Omega} u^{q} \phi_1 dx = \int_{\Omega} \sqrt{\lambda_1} v \phi_1 dx.
\end{eqnarray*}
We use a convex inequality to get
\begin{eqnarray*}
\int_{\Omega} \lambda_1 u \phi_1 dx \geq C (\int_{\Omega} v \phi_1 dx)^{p}
\quad \textrm{and}\quad \int_{\Omega} \lambda _1 v \phi_1 dx \geq C (\int_{\Omega} u \phi_1 dx)^{q},
\end{eqnarray*}
which yields that
\begin{eqnarray*}
\int_{\Omega} v \phi_1 dx \leq C \quad \textrm{and} \quad \int_{\Omega} u \phi_1 dx \leq C.
\end{eqnarray*}
We also have
\begin{eqnarray*}
\int_{\Omega} ( v^p + u^{q} ) \phi_1 dx \leq C.
\end{eqnarray*}
From \eqref{yand}, Lemma \ref{uvc2} and the above inequality, we can obtain  $L^{\infty}$-bound for $(u,v)$ near the boundary $\partial \Omega$. Thus we obtain
\begin{eqnarray*}
\int_{\Omega} ( v^{p+1} + u^{q+1}) dx \leq C.
\end{eqnarray*}
We now use the bootstrap argument to improve the integrability of $v$ and $u$. For this, we need
\begin{eqnarray*}
\frac{p}{(p+1)\rho^{i}} - \frac{1}{(q+1) \rho^{i+1}} < \frac{1}{n}\quad \textrm{and} \quad \frac{q}{(q+1)\rho^i} - \frac{1}{(p+1)\rho^{i+1}} < \frac{1}{n}.
\end{eqnarray*}
It is enough to get
\begin{eqnarray*}
\frac{p}{p+1} - \frac{1}{(q+1)\rho} < \frac{1}{n} \quad \textrm{and} \quad \frac{q}{q+1} - \frac{1}{(p+1) \rho} < \frac{1}{n}.
\end{eqnarray*}
We need to choose $\rho$ so that
\begin{eqnarray*}
\frac{1}{\rho} > \max \left[ (q+1) \left( \frac{n-2s}{n} - \frac{1}{p+1}\right) , (p+1) \left( \frac{n-2s}{n} -\frac{1}{q+1}\right)\right].
\end{eqnarray*}
Because $\frac{1}{p+1} + \frac{1}{q+1} = \frac{n-2s}{n} + \epsilon$, we may choose $\rho$ so that
\begin{eqnarray*}
\frac{1}{\rho} > 1 - \epsilon \min(p+1, q+1).
\end{eqnarray*}
This enables the bootstrapping. The proof is complete.
\end{proof}

\section*{Acknowledgements}\thispagestyle{empty}
I wish to thank my advisor Prof. Rapha\"el Ponge for his support and valuable suggestions during the preparation of this paper.

\end{document}